\documentclass[a4paper,11pt]{article}
\usepackage[utf8x]{inputenc}
\usepackage{kpfonts} 
\usepackage[T1]{fontenc} 
\usepackage{diagbox}
\usepackage{a4wide}
\usepackage{amssymb,amsmath,xspace}
\usepackage{amsfonts}
\usepackage{thm-restate} 
\usepackage{fixmath} 
\usepackage[scr]{rsfso}
\usepackage{nicefrac} 
\usepackage[pdftex]{color,graphicx}
\usepackage{latexcolors} 
\usepackage[pdftex,colorlinks=true,pdfstartview=FitV, linkcolor=egyptianblue,citecolor=egyptianblue,urlcolor=egyptianblue]{hyperref}
\usepackage[capitalize]{cleveref} 
\usepackage{enumitem} 
\setlist[itemize]{label=\textbullet} 
\usepackage{multirow} 
\usepackage{xparse} 
\pdfoutput=1 

\newenvironment{myfigure}%
  {\begin{figure}[htbp!]\centering}%
  {\end{figure}}

\def\myfootnotemark#1#2{%
{\let\thefootnote\relax\footnotemark\addtocounter{footnote}{-1}\hspace{-.2ex}}%
\textsuperscript{\hyperref[#1:#2]{\the\numexpr\thefootnote+#2}}%
}
\def\myfootnotetext#1#2#3{%
\addtocounter{footnote}{1}%
\footnotetext{\label{#1:#2}#3}%
}

\newtheorem{claim}{Claim}
\newtheorem{lemma}{Lemma}

\newtheorem{proposition}{Proposition}

\newtheorem{corollary}{Corollary}

\newenvironment{proof}{\noindent\textbf{Proof.}}{{}\hfill$\Box$\\}
\newenvironment{proofof}{\noindent\textbf{Proof of~}}{{}\hfill$\Box$\\}

\def\leq{\leqslant}\def\le{\leq}
\def\geq{\geqslant}\def\ge{\geq}
\def\emptyset{\varnothing}

\def\eps{\varepsilon}

\NewDocumentCommand\set{sm}{\IfBooleanTF#1{\{{#2}\}}{\left\{{#2}\right\}}}
\NewDocumentCommand\ceil{sm}{\IfBooleanTF#1{\lceil{#2}\rceil}{\left\lceil{#2}\right\rceil}}
\NewDocumentCommand\floor{sm}{\IfBooleanTF#1{\lfloor{#2}\rfloor}{\left\lfloor{#2}\right\rfloor}}
\NewDocumentCommand\pare{sm}{\IfBooleanTF#1{({#2})}{\left({#2}\right)}}
\NewDocumentCommand\range{smm}{\IfBooleanTF#1{\set*{{#2},\dots,{#3}}}{\set{{#2},\dots,{#3}}}}
\NewDocumentCommand\card{sm}{\IfBooleanTF#1{|{#2}|}{\left|{#2}\right|}}

\NewDocumentCommand\cardV{sm}{\IfBooleanTF#1{|V({#2})|}{\card{V\!\pare{#2}}}}

\def\Bdiv#1#2{\binom{#1}{2}{\big/}\binom{#2}{2}}

\newcommand{\cH}{\mathcal{H}}

\newcommand{\sC}{\mathscr{C}}
\newcommand{\sF}{\mathscr{F}}
\newcommand{\sG}{\mathscr{G}}
\newcommand{\sK}{\mathscr{K}}
\newcommand{\sP}{\mathscr{P}}

\newcommand{\sT}{\mathscr{T}}
\newcommand{\sU}{\mathscr{U}}

\usepackage{tikz} 
\usetikzlibrary{arrows.meta}
\usetikzlibrary{decorations.pathreplacing,calligraphy}
\usetikzlibrary{fit,positioning}
\usetikzlibrary{graphs} 
\usetikzlibrary{calc} 

\tikzstyle{vertex}=[draw=black,fill=white,circle,inner sep=1.4pt]

\usepackage{tikzit}

\tikzstyle{circle}=[fill=white, draw=black, shape=circle]
\tikzstyle{green_node}=[fill=green!20, draw=black, shape=circle]
\tikzstyle{red_node}=[fill=red!20, draw=black, shape=circle]
\tikzstyle{blue_node}=[fill=blue!20, draw=black, shape=circle]
\tikzstyle{yellow_node}=[fill=yellow, draw=black, shape=circle]
\tikzstyle{pink_node}=[fill={rgb,255: red,251; green,15; blue,192}, draw=black, shape=circle]
\tikzstyle{doted_node}=[fill=white, draw=black, shape=circle, dashed]
\tikzstyle{red_square}=[fill=red!20, draw=black, shape=rectangle]
\tikzstyle{green_square}=[fill=green!20, draw=black, shape=rectangle]
\tikzstyle{yellow_square}=[fill=yellow, draw=black, shape=rectangle]
\tikzstyle{ellipse}=[fill=white, draw=black, shape=ellipse]
\tikzstyle{orange_node}=[fill={rgb,255: red,255; green,128; blue,0}, draw=black, shape=circle]
\tikzstyle{blue_square}=[fill=blue!20, draw=black, shape=rectangle]

\tikzstyle{edge}=[-]
\tikzstyle{purple_edge}=[-, draw=magenta]
\tikzstyle{blue_edge}=[-, draw={rgb,255: red,30; green,52; blue,255}]
\tikzstyle{orange_edge}=[-, draw={rgb,255: red,255; green,128; blue,0}]
\tikzstyle{doted_edge}=[-, tikzit draw=black, dashed]
\tikzstyle{green_edge}=[-, draw=green]
\tikzstyle{blue_doted}=[-, draw=blue, dashed]
\tikzstyle{red_edge}=[-, draw=red]
\tikzstyle{arrow}=[->, draw=black]
\tikzstyle{doted purple}=[-, draw=magenta, dashed]

\title{\textbf{Lower Bounds for Induced-Universal Graphs}}

\author{
  Cyril Gavoille\thanks{\textsf{gavoille@labri.fr}.
  }
  \and
  Amaury Jacques\thanks{\textsf{amaury.jacques@labri.fr}}%
}

\date{%
  LaBRI, University of Bordeaux, France\\[2ex]
  \today%
}

\usepackage{multicol}
\begin{document}

\maketitle

\begin{abstract}
    We give a series of new lower bounds on the minimum number of vertices required by a graph to contain every graph of a given family as induced subgraph. In particular, we show that this induced-universal graph for $n$-vertex planar graphs must have at least $10.52n$ vertices. We also show that the number of conflicting graphs to consider in order to beat this lower bound is at least~$137$. In other words, any family of less than~$137$ planar graphs of $n$ vertices has an induced-universal graph with less than $10.52n$ vertices, stressing the difficulty in beating such lower bounds. Similar results are developed for other graph families, including but not limited to, trees, outerplanar graphs, series-parallel graphs, $K_{3,3}$-minor free graphs. As a byproduct, we show that any family of $t$ graphs of $n$ vertices having small chromatic number and sublinear pathwidth, like any proper minor-closed family, has an induced-universal graph with less than $\nicefrac{15}{7} \sqrt{t} \cdot n$ vertices. This is achieved by making a bridge between equitable colorings, combinatorial designs, and path-decompositions.
    
\end{abstract}



\section{Introduction}

Universality plays an important role in Graph Theory and Computer Science. For the famous Traveling Salesperson Problem, \cite{JLNRS05} showed that every set $S$ of $n$ points in a metric space has some universal tour, computable in polynomial time, that approximates, up to poly-logarithmic factor, the best tour of every subset of $S$ by taking the induced sub-tour of this universal ordering. In graph theory, the celebrated Excluded Grid Theorem of Robertson and Seymour~\cite{RST94} (which states that graphs excluding as minor a fixed planar graph have bounded treewidth), relies on the fact that every planar graph is a minor of some relatively small grid. So, excluding such a grid as minor implies excluding as minor every sufficiently small planar graph. Therefore, rather than considering individually each graph of a given graph family, it is much simpler to manipulate one single graph (here a grid), whose properties is close enough to those of the graphs of the family (here the property is planarity). This grid-minor universality has been extended by \cite{GH23a}, by showing that for every $n$ and $g$, there is some fixed graph of Euler genus $g$ and with only $O(g^2(n+g)^2)$ vertices that is minor-universal for all $n$-vertex graphs of Euler genus $g$. Here the size of the minor-universal graph is an important parameter, and its polynomial dependency in $n$ and $g$ is the crux for the very recent polynomial bound for the Graph Minor Structure Theorem~\cite{GSW25}.

\paragraph{Induced-universal graphs.}

In this paper, we focus on universality under induced-subgraph containment. More precisely, a graph $U$ is an \emph{induced-universal graph} for a graph family $\sF$ if every graph of $\sF$ is isomorphic to an induced-subgraph of $U$. Here, the term "graph family" must be considered in a broad sense: it can contain a finite set of graphs (like two bounded size graphs, or two trees with $n$ vertices) or a countably finite set of graphs (like the class of all planar graphs). We denote by $\sU(\sF)$ the smallest number of vertices of an induced-universal graph for $\sF$, and by $\sF_n$ the sub-family composed only of $n$-vertex graphs of $\sF$ (if any). 

For the family $\sG$ of all graphs, Alon~\cite{Alon17} showed that $\sU(\sG_n) = (1+o(1)) \cdot 2^{(n-1)/2}$. For the family $\sP$ of all planar graphs, a result of \cite{DEGJx21,GJ22} implies that $\sU(\sG_n) \le n \cdot 2^{O(\!\sqrt{\log{n}}\,)}$. And for the family $\sT$ of all trees, \cite{ADBTK17} showed that $\sU(\sT_n) \le cn$, where $c\gg 768$ is a rather large constant\footnote{It is not given explicitly in~\cite{ADBTK17}. Their involved construction is based on a solution for caterpillars that contains at least $12\cdot 2^5 n = 384n$ vertices. Extending the solution for trees requires to at least double the caterpillar solution. An estimate of $10^5$ or $10^6$ for $c$ seems to be closer to reality.}.
The results for trees and planar graphs come from the design of \emph{adjacency labeling schemes}, which can be viewed as an alternative definition of induced-universal graphs with algorithmic applications~\cite{KNR92}.

Whereas for all graphs, the first order of $\sU(\sG_n)$ is well-established, the situation is very different for planar graphs and trees. Up to the best of our knowledge, the only non-trivial better-than-$n$ lower bound is $\sU(\sP_n) \ge 11\floor{n/6} \approx 1.83n$, coming from the family of paths and cycles~\cite[Claim~1]{ELO08}. For trees, the lower bound is even weaker. A folklore bound $\sU(\sT_n) \ge 3\floor{n/2}$ can be obtained by considering a family of two trees that can only share $\ceil{n/2}$ vertices: a star (a depth-1 tree) and a path, each with $n$ vertices. These both lower bounds are ad-hoc techniques based on the impossibility of coexisting a certain number of "conflicting graphs" in a single graph that is too small.

\paragraph{On the difficulty of proving lower bounds.}

It is perhaps worth mentioning that it can be very difficult to prove strict lower bounds on the size of universal objects for specific combinatorial properties. For instance, consider \emph{universal planar point sets}, i.e., a fixed set of points such that each $n$-vertex planar graph has a straight-line drawing on it. The best upper bound is $2n^2$ (an $2n\times n$-grid), and the current best lower bound has been improved only recently to $1.293n$ for $n$ large enough~\cite{SSS20}, a large gap. Interestingly, \cite{SSS20} showed that for $n = 11$, no universal planar point sets of exactly~$n$ points exist. In other words, the lower bound is at least $n+1$ for $n = 11$. Proving an $(n+1)$-lower bound, even for small $n$, is not so easy. Technically, they have considered (computer-assisted) a family of~$49$ "conflicting" graphs, each with~$n = 11$ vertices, which cannot be simultaneously drawn on any set of~$n$ points. They show that more than~$36$ conflicting graphs are required. In other words, for $n = 11$, in order to beat the trivial $n$-lower bound any conflicting family has size at least~$37$. It is only conjectured that the $n$-lower bound can be beaten for every $n\ge 11$. There are even conjectures made about the size of any conflicting family w.r.t the trivial $n$-lower bound~\cite[(3), p.~129]{BCDEx07}.

For induced-universal graphs for trees, it is not clear if a lower bound as low as say~$2.1n$ can be achieved. As previously said, the folklore $3\floor{n/2}$-lower bound for trees is achieved by considering a conflicting family 
of two caterpillars\footnote{Recall that a caterpillar is a tree in which the nodes of degree at least two induced a path. They are also exactly the maximal graphs of pathwidth one.} on $n$ vertices: a star and a path. 
Naively, one may believe that by considering a conflicting family of three trees (instead of two caterpillars) may allow us to easily beat this folklore $3n/2$-lower bound. Unfortunately, one of our results implies that $\sU(\sF) \le 3n/2 + O(\log{n})$ for any family $\sF$ composed of three trees with $n$ vertices. So, at least~$4$ trees need to be considered in any conflicting family w.r.t. the $3n/2$-lower bound. Even worse, we show that in order to beat~$2.1n$, we will need to consider at least~$7$ trees. Clearly, a lower bound proof involving~$7$ trees is far more complicated than one involving~$2$ trees, as pairwise intersections of the trees in an induced-universal graph solution may be hard to control.

\paragraph{Our contributions.}

First, we use some classical Information-Theoretic lower bound to improve the current lower bounds based on conflicting families. This concerns trees, planar and some other families (see \cref{tab:contrib} and \cref{app:tab:lb_subgraphs} columns~$c$). We also develop another lower bound for specific conflicting families containing small cliques (see \cref{app:tab:lb_cliques}). The bound is actually optimal for unions of cliques of bounded size. This captures the folklore $3n/2$-lower bound for trees, and the $11n/6$-lower bound for paths and cycles. As intermediate step, it gives a simple $25n/12$-lower bound for planar graphs by considering cliques of size at most four. Using our Information-Theoretic lower bound, the bound for planar graphs is improved to $10.52n$.

Second, we use block designs and special vertex-coloring to construct small induced-universal graphs for any family $\sF$ with $t$ graphs of $n$ vertices. The result depends on two more integral parameters on $\sF$: $p$ and $k$. More precisely, we show in \cref{th:universal} that, if every graph of $\sF$ has a set of $p$ vertices whose removal leaves a graph that is \emph{equitably $k$-colorable} (a $k$-coloring with same color-class sizes), then
\begin{equation}\label{eq:U}
\sU(\sF) ~\le~ s \cdot \frac{n}{k} + tp ~.
\end{equation}
The number $s = s(t,k)$ is related to some block designs and the Coding-Theoretic Function as defined in \cref{sec:block_design}. Actually, we show in \cref{th:pathwidth} that parameter $p$ is at most the pathwidth of any graph in $\sF$ times the chromatic number $k$ of $\sF$.

For instance, for any family $\sF$ composed of $t = 3$ trees with $n$ vertices, Eq.~\eqref{eq:U} implies $\sU(\sF) \le 3n/2 + O(\log{n})$, since $k = 2$, $p = O(\log{n})$ (trees are bipartite of logarithmic pathwidth), and because $s = 3$ in this case. For $\sF$ composed of $t = 6$ trees, we have $s = 4$, and thus $\sU(\sF) \le 2n + O(\log{n})$. If $\sF$ is composed of $t=20$ planar graphs, $\sU(\sF) \le 4n + O(\!\sqrt{n}\,)$, since in that case $k = 4$, $p = O(\!\sqrt{n}\,)$, and $s = 16$.

From above, at least~$7$ conflicting trees would be required in order to beat a hypothetical $2.1n$-lower bound for trees, as $\sU(\sF) \le 2n + O(\log{n})$ for $t = 6$ trees. Actually, we have showed a lower bound of $1.626n$ for tree families, and $\sU(\sF) \le 3n/2 + O(\log{n})$ for $t = 3$ trees. Thus, at least~$4$ conflicting trees are required in order to beat the $1.626n$-lower bound (cf. the first line of \cref{tab:contrib}). For planar graph families, $k = 4$, $p = O(\!\sqrt{n}\,)$, and $s = 42$ if $t = 136$. So, at least~$137$ conflicting graphs are required in order to beat the $10.52n$-lower bound.

This framework applies to several families of graphs, and the series of new lower bounds we obtained, on $\sU$ and conflicting family size, are summarized in columns $c$ and $t$ of \cref{tab:contrib}.
%
%
%
%
%
%
%


\begin{table}[htbp!]
\begin{center}
\renewcommand{\arraystretch}{1.05}
\begin{tabular}{cc||c|c||clc}
&$\sF$ & $c$& $t$ & $k$ & $A(\ceil{ck}-1,2k-2,k)$&\\
\hline
&forests & $\phantom{0}1.626$ & $4$ & $2$ &$A(3,2,2) = 3$&\\ 
&outer-planar & $\phantom{0}3.275$& $13$ & $3$& $A(9,4,3) = 12$&\\
&series-parallel & $\phantom{0}3.850$ & $17$ & $3$& $A(11,4,3) = 16$&\\ 
&$K^-_5$-minor-free\myfootnotemark{foot}{1} & $\phantom{0}6.264$ & $51$ & $4$ & $A(25,6,4) = 50$&\\
&planar & $10.520$ & $137$ & $4$ & $A(42,6,4) = 136$&\\
&$K_{3,3}$-minor-free & $10.521$ & $124$ & $5$ & $A(52,8,5) \in\set{123,124}$\myfootnotemark{foot}{2}&\\
\hline
\end{tabular}
\caption{New lower bounds given by \cref{th:lb_subgraphs} for various family $\sF$ of the form $\sU(\sF_n) \ge cn - o(n)$. Here $t$ is a lower bound on the size of a conflicting family needed to beat the corresponding $cn$-lower bound. From \cref{th:conflict_size}, we have $t = 1 + A(\ceil{ck}-1,2k-2,k)$, where $k$ is the chromatic number of $\sF$, and $A$, related to block designs, is the Coding-Theoretic Function defined in \cref{sec:block_design}.}
\label{tab:contrib}
\end{center}
\end{table}
\myfootnotetext{foot}{1}{$K^-_5$ is the graph $K_5$ minus one edge.} \myfootnotetext{foot}{2}{According to \href{https://aeb.win.tue.nl/codes/Andw.html\#d8}{Andries E. Brouwer's web pages}, it is not known whether $A(52,8,5) = 123$ or $124$.}

Motivated by the quest for an $\Omega(n\log{n})$-lower bound for planar graph, we show that a conflicting family w.r.t. to this lower bound (if it exists) must have a size of at least $\Omega(\log^2{n})$. More generally, we show in \cref{th:asuniversal} that any family $\sF$ of $t$ graphs having a small chromatic number and sublinear pathwidth, like any proper minor-closed family as planar or bounded genus graphs, verifies $\sU(\sF_n) < \nicefrac{15}{7}\sqrt{t}\cdot n$.

\paragraph{Organization of the paper.}

\Cref{sec:lower_bounds} presents the union-of-clique lower bound, and the Information-Theoretic lower bound. \Cref{sec:almosteq} gives our construction, as in Eq.~\eqref{eq:U}, using block designs, and its applications to conflicting families. \Cref{sec:towards} presents our upper bound in $O(\!\sqrt{t}\cdot n)$.



\section{Lower Bounds}
\label{sec:lower_bounds}

\subsection{Union of cliques}

The first lower bound is based on the union of small cliques. The bounds presented in \cref{th:lb_cliques} are optimal and generalize previously known lower bounds: the folklore lower bound of $\floor{3n/2}$ for acyclic graphs of maximum degree two, and the lower bound of $11\floor{n/6}$ for graphs of maximum degree at most two~\cite[Claim~1]{ELO08}. We remark for this latter that our lower bounds are slightly better. It actually matches ours only for $n$ multiple of six. E.g., if $n = 6i+5$ for integer $i$, $11\floor{n/6} = 11i$ whereas our bound gives $11i+8$.

\begin{restatable}{theorem}{thlbcliques}\label{th:lb_cliques}
  Let $\sK_{n,k}$ be the family of graphs with at most $n$ vertices composed of a disjoint union of cliques each with at most $k$ vertices. Then, $\sU(\sK_{n,k}) = \sum_{i=1}^k \floor{n/i} \ge (n+1) \ln{(k+1)} - k$.
\end{restatable}

\begin{proof}
    \vspace{-3ex}\paragraph{Lower bound.}

    \def\tG{\tilde{G}}%
    For every $i\in\range{1}{k}$, let $G_i$ be the $n$-vertex graph composed of $\floor{n/i}$ disjoint cliques of $i$ vertices plus $n \bmod i$ isolated vertices, and let $F_i = \range{G_1}{G_i}$. Clearly, $F_i \subseteq \sK_{n,i}$, and thus in particular $\sU(\sK_{n,k}) \ge \sU(F_k)$. Our goal is to show that $\sU(F_k) = \sum_{i=1}^k \floor{n/i}$. For convenience, let $N_i = \sU(F_i)$. Obviously, $N_1 = n$ since $F_1 = \set{G_1}$ and $G_1$ has $n$ vertices. Assume $k\ge 2$. We will show that $N_k \ge \floor{n/k} + N_{k-1}$.
    
    Let $U$ be a minimum induced-universal graph for $F_k$, thus with $N_k$ vertices. We construct as follows a subgraph $U'$ of $U$ that is an induced-universal graph for $F_{k-1}$ with no more than $N_k - \floor{n/k}$ vertices. First, consider any induced-subgraph embedding of $G_k$ in $U$, and denote by $C_1,\dots,C_{\floor{n/k}}$ be all its cliques embedded in $U$ with $k\ge 2$ vertices. Thus, the subgraph $\bigcup_{i=1}^{\floor{n/k}} C_i$ is an induced union of cliques of $U$, each with $k\ge 2$ vertices. Second, for each $C_i$, select an arbitrary vertex $x_i$. Finally, we set $U' = U\setminus X$, where $X = \range*{x_1}{x_{\floor{n/k}}}$. Note that $X$ is a stable set in $U$, because there are no edges in $U$ between any two distinct cliques of $\bigcup_{i=1}^{\floor{n/k}} C_i$.
    
    Clearly, $\cardV{U'} = \cardV{U} - |X| = N_k - \floor{n/k}$. It remains to show that $U'$ is an induced-universal graph for $F_{k-1}$. 
    Consider $G_t \in F_{k-1}$, for some $t<k$, and denote by $D_1,D_2,\dots$ be all its clique components that are embedded in $U$. The subgraph $\bigcup_{i} D_i$ is an induced union of cliques of $U$, each with $1$ or $t$ vertices. Each clique $D_i$ can intersect $X$ in at most one vertex, because $X$ is a stable set. Suppose that $D_i$ intersects $X$ in vertex $x_j$ for some $i,j$. The key point is that the clique $C_j$ containing $x_j$ intersects no other $D_m$ with $m\neq i$. Otherwise, $D_i \cup D_m$ would be not an induced union of disjoint cliques in $U$. Also, we have $\cardV*{C_j} = k > t = \cardV{D_i}$. It follows that, for each such clique $D_i$, one can move its embedding into $C_j \setminus\set*{x_j}$ instead of $D_i$ while preserving that their union is disjoint. One can check that their union is also induced in $U$ if $U'$ is \emph{minimal}, i.e., with no proper subgraphs that is induced-universal for the same family. So, this new embedding is preserved as well in $U\setminus\set*{x_j}$, since $x_j$ is not useful anymore. By applying this embedding transformation for each clique $D_i$ intersecting $X$, we have shown that $G_t$ is induced in $U' = U\setminus X$. This applies for each $G_t \in F_{k-1}$.
    
    Therefore, $U'$ is an induced-universal for $F_{k-1}$, and thus $N_{k-1} \le \cardV{U'} \le N_k - \floor{n/k}$. In other words, we have established that:
    \begin{eqnarray*}
        \forall k\ge 2, \quad N_k &\ge& \floor{\frac{n}{k}} + N_{k-1} ~.
    \end{eqnarray*}
    By iterating the previous inequality down to $k=2$, and using the fact that $N_1 = n$, we get:
    \[
        \sU(\sK_{n,k}) ~\ge~ N_k ~\ge~ \sum_{i=1}^k \floor{\frac{n}{i}} ~.
    \]
    We remark\footnote{\label{foot:a/b}Using the fact that $\floor{a/b}+1 \ge (a+1)/b$ for integers $a$ and $b \neq 0$.} that $\sum_{i=1}^k \floor{n/i} \ge \sum_{i=1}^k ((n+1)/i - 1) = (n+1) \cdot (\sum_{i=1}^k 1/i) - k \ge (n+1) \cdot \int_{1}^{k+1} \frac{1}{x} dx - k$, and thus $\sU(\sK_{n,k}) \ge (n+1)\ln{(k+1)} -k$.
    
    \paragraph{Upper bound.}
    
    For every  $i\in\range{1}{k}$, we inductively construct a graph $U_i$ as follows. For $i=1$, we let $U_1$ be a stable set with $n$ vertices. Then, for $i\ge 2$, assuming that $U_{i-1}$ contains $\floor{n/(i-1)}$ disjoint cliques of $i-1$ vertices, the graph $U_i$ is obtained from a copy of $U_{i-1}$ in which we replace $\floor{n/i}$ cliques of $i-1$ vertices each by a clique of $i$ vertices. In other words, we increase by one vertex $\floor{n/i}$ cliques taken among the $\floor{n/(i-1)}$ cliques of $i-1$ vertices of $U_{i-1}$. So, $U_i$ has exactly $\floor{n/i}$ vertices more than $U_{i-1}$, and it contains $\floor{n/i}$ disjoint cliques of $i$ vertices. Moreover, $U_i$ contains $U_j$ as induced subgraph for each $j\le i$. Thus, $U_i$ contains $\floor{n/j}$ disjoint cliques of at least $j$ vertices.
    
    In particular, $\cardV{U_k} = \floor{n/k} + \cardV{U_{k-1}} = \sum_{i=1}^k \floor{n/k}$, and $U_k$ contains $\floor{n/j}$ disjoint cliques of at least $j$ vertices, for every $j\in \range{1}{k}$. Also observe that $U_k$ has a total of $n$ disjoint cliques. Denoted by $C_i$ be the $i$th largest clique of $U_k$.
    
    Let us show that $U_k$ is induced-universal for $\sK_{n,k}$. Consider a graph $G \in \sK_{n,k}$, and let $D_1,\dots,D_t$ be its $t\le n$ disjoint cliques ordered by non-increasing size. The cliques of $G$ are embedded in $U_k$ in a greedy way: $D_i$ is embedded into $C_i$, for each $i = 1,2,\dots,t$.
    
    To show that this embedding does not fail, it is sufficient to show that $\cardV{C_i} \ge \cardV{D_i}$. Let $d_i = \cardV{D_i}$. Since $G\in \sK_{n,k}$, $d_i \in\range{1}{k}$. The main observation is that $d_i \le \cardV{G}/i \le n/i$. This is because cliques of $G$ are considered by non-increasing order. We have seen that $U_k$ contains $\floor{n/j}$ disjoint cliques with at least $j$ vertices for every $j \in \range{1}{k}$. So, since $d_i\le n/i$, $U_k$ contains $\floor{n/d_i} \ge \floor{n/(n/i)} = i$ disjoint cliques with at least $d_i$ vertices. Therefore, the $i$th largest cliques in $U_k$ has at least $d_i$ vertices, that is $\cardV{C_i} \ge d_i = \cardV{D_i}$ as claimed.
\end{proof}

\begin{myfigure}
    \def\A#1{
        \B{#1}
        \node (z) at ($(u)+(0.5,0.30)$) {};
        \draw (u) -- (z) -- (v);
        \draw (z) -- (w);
    }
    \def\B#1{
        \C{#1}
        \node (w) at ($(u)+(0.5,0.78)$) {};
        \draw (v) -- (w) -- (u);
    }
    \def\C#1{
        \node (u) at #1 {};
        \node (v) at ($(u)+(1,0)$) {};
        \draw (u) -- (v);
    }
    \begin{tikzpicture}
    \tikzset{every node/.style={vertex}}
    \foreach \i in {0,1,2,3,4,5} { \A{(1.5*\i,0)} }
    \foreach \i in {0,1} { \B{(1.5*\i,-1)} }
    \foreach \i in {2,3,4,5} { \C{(1.5*\i,-1)} }
    \foreach \i in {0,1,2,3,4,5} { \node at (1.5*\i,-1.5) {}; }
    \foreach \i in {0,1,2,3,4,5} { \node at (1.5*\i,-2) {}; }
  \end{tikzpicture}
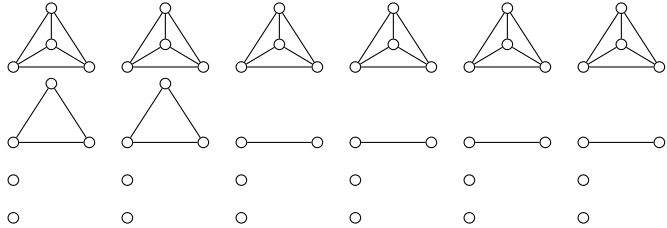
  \caption{The minimum induced-universal graph for $\sK_{24,4}$, the family of graphs with at most~$24$ vertices that are union of disjoint cliques with at most $4$ vertices. Its has $\sU(\sK_{24,4}) = 24 \cdot (1 + 1/2 + 1/3 + 1/4) = 50$ vertices.}
  \label{fig:K24_4}
\end{myfigure}

Note that graphs in $\sK_{n,2}$ are forests (it is actually a matching), graphs in $\sK_{n,3}$ have maximum degree at most two and thus are outer-planar, and $\sK_{n,4}$ are planar. Thus, for these families, for infinitely many $n$ (for $n$ divisible by $k!$), we have respectively 
$3n/2$, $11n/6$, and $25n/12$ as lower bounds on the number of vertices for their minimum induced-universal graphs (see \cref{app:tab:lb_cliques}). However, for forests, outer-planar and planar graphs, the next lower bounds, as presented in \cref{th:lb_subgraphs}, will be better (see also \cref{app:tab:lb_subgraphs}).

\begin{table}[htpb!]
\begin{center}
\renewcommand{\arraystretch}{1.05}
\begin{tabular}{cc|ll}
& $\sF$ & lower bounds on $\sU(\sF_n)$&\\
\hline
&caterpillar forests $\supset \sK_{n,2}$ &  $3n/2$&\\ 
&union of cycles and paths $\supset \sK_{n,3}$ & $11n/6$&\\
&planar $\supset \sK_{n,4}$ & $25n/12$&\\
&union of cliques with $\le k$ vertices $= \sK_{n,k}$ & $(n+1)\ln(k+1) - k$&\\
&union of cliques $= \sK_{n,n}$ & $(n+1)\ln(n+1) - n$&\\
\hline
\end{tabular}
\caption{Summary of lower bounds derived from \cref{th:lb_cliques}.}
\label{app:tab:lb_cliques}
\end{center}
\end{table}

\subsection{Information-Theoretic lower bound}

Given an infinite graph family $\sF$, the \emph{unlabeled constant-growth} of $\sF$ is the following number (if it exists),
\[
  g(\sF) ~=~ \limsup_{n\to \infty}\pare{u_n}^{1/n}
\]
where $u_n$ counts the number of unlabeled graphs of $\sF$ having $n$
vertices.

\begin{restatable}{theorem}{thlbsubgraphs}\label{th:lb_subgraphs}
  Let $\sF$ be a infinite graph family with unlabeled constant-growth at least\footnote{Here $e = \exp(1) = 2.7182818284...$.} $g > e/2$.
  For every $0 < \eps < g/e - \nicefrac{1}{2}$, and for infinitely many $n$, $\sU(\sF_n) > (c-\eps)n$, where $c$ is solution of $c^c / (c-1)^{c-1} = g$. We have $c \ge g/e + \nicefrac{1}{2}$. Moreover, if $\lim_{n\to \infty} (u_n)^{1/n}$ does exist, then $\sU(\sF_n) \ge cn$ for every large enough $n$.
\end{restatable}

\begin{proof}
Consider any minimum induced-universal graph $U$ for $\sF_n$, let $N = \sU(\sF_n)$ be its number of vertices, and let $u_n$ be the number of unlabeled graphs in $\sF_n$.


From the universality of $U$, every graph of $\sF_n$ is isomorphic to an $n$-vertex induced subgraph of $U$. So, the number of pairwise non-isomorphic $n$-vertex induced subgraphs in $U$ must be an upper bound on $u_n$. Such induced subgraphs can be obtained by selecting a subset of $n$ vertices taken among the $N$ of $U$. Thus, there are at most $\binom{N}{n}$ induced subgraphs in $U$ having $n$ vertices. It follows that $u_n \le \binom{N}{n}$.

We have the well-known upper bound for the binomial numbers (see~\cite[p.~255]{Bollobas78} for instance):
\[
\forall\, 0 < b < a, \quad \binom{a}{b} ~<~ \pare{\frac{a}{b}}^b \pare{\frac{a}{a-b}}^{a-b} = \pare{\alpha\pare{\frac{\alpha}{\alpha-1}}^{\alpha-1}}^b ~=~ f(\alpha)^b
\]
where $\alpha = a/b > 1$ and $f: x\mapsto x^x / (x-1)^{x-1}$ is related to the entropic-function. 

Assume $N \le cn$ for all $n$ and some constant $c>1$. We have, for every $n\ge 1$:
\begin{eqnarray*}
    u_n &\le& \binom{N}{n} ~\le~ \binom{cn}{n} ~<~ f(c)^n\\
    \Rightarrow \quad \pare{u_n}^{1/n} &<& f(c) ~.
\end{eqnarray*}
The latter inequality implies that $\limsup_{n\to \infty} (u_n)^{1/n} \le f(c)$, since\footnote{This is because, for every sequence $x_n$, $\limsup_{n\to \infty} x_n = \inf_{n\ge 0}\set{\sup_{m\ge n} x_m}$.} $f(c)$ is a constant independent of $n$. In other words, if $N\le cn$ for every $n$ and $c>1$, then $g \le g(\sF) \le f(c)$, where $g$ is a lower bound on the unlabeled constant-growth of $\sF$. A basic property of $\limsup$ is that, for every sequence $x_n$, if $\limsup_{n\to\infty} x_n > \ell$, for some constant $\ell$, then, $x_n > \ell$ for infinitely many $n$. It turns out that if $g > f(c)$, then $(u_n)^{1/n} > f(c)$, and thus $N > cn$ for infinitely many $n$.

One can check that the derivative of $f$ is $f'(x) = f(x) \ln{(x/(x-1))} > 0$ for $x>1$,
and that $f(x) \to ex - e/2$ when $x\to\infty$ (using the facts that $\lim f'(x) = e$ and $\lim f(x) - ex = -e/2$). In fact, analyzing the second derivative of $f$, we have $f(x) \le ex - e/2$ for all $x > 1$. 
In particular, $f(x)$ is non-decreasing for $x > 1$, and the real solution of the equation $f(c) = g$ is some $c_0$ satisfying $g \le e c_0 -e/2$, i.e., $c_0 \ge g/2 + \nicefrac{1}{2}$. 

We have $0 < \eps < g/e - \nicefrac{1}{2}$ by assumption. Let $c = c_0 - \eps$. Since $c_0 \ge g/2 + \nicefrac{1}{2}$, we have $c \ge g/e + \nicefrac{1}{2} - \eps > g/e + \nicefrac{1}{2} - (g/e - \nicefrac{1}{2}) > 1$. So, $f(c) = f(c_0 - \eps) < g$ as $f(x)$ is non-decreasing for $x>1$, which in turns implies that is $\sU(\sF_n) = N > (c_0-\eps)n$ for infinitely many $n$.

Now assume that $(u_n)^{1/n}$ has a limit, which must be $g(\sF)$. Following the same argument as above, if $N \le cn$ for every large enough $n$, then $(u_n)^{1/n} < f(c)$. In particular, $\lim_{n\to\infty} (u_n)^{1/n} \le f(c)$, and thus $g \le g(\sF) \le f(c)$. In other words, if $g \ge f(c)$, then $N \ge cn$ for every large enough $n$. Therefore, for $c$ solution of $f(c) = g$, we indeed have $N \ge cn$ for every large enough $n$.

This completes the proof.
\end{proof}

At first glance, the condition "for infinitely many $n$" in the lower bound of \cref{th:lb_subgraphs} is not satisfactory, as we may expected stronger condition like "for every $n$, $\sU(\sF_n) \ge (c-\eps)n$" or "$\sU(\sF_n) \ge cn - o(n)$". However, we observe that there are graph families where such a stronger condition cannot be true. This is the case of cubic planar graphs since there are no cubic graphs with an odd number of vertices. Nevertheless, the lower bound $\sU(\sF_n) \ge cn - o(n)$ holds for all the families considered in Table~\ref{app:tab:lb_subgraphs}, as $\lim_{n\to\infty} (u_n)^{1/n}$ exists for all of them.


Now, using known results on counting graphs for various families, and combined with \cref{th:lb_subgraphs}, we get the following set of new lower bounds (cf. \cref{app:tab:lb_subgraphs}):

\begin{table}[ht]
\begin{center}
\renewcommand{\arraystretch}{1.05}
\begin{tabular}{cc|cc|cc}
&$\sF$ & $g$ & refs & $c$&\\
\hline
&caterpillar forests\protect\footnotemark{} & $2\phantom{.000*}$ &\cite{HS73}& $\phantom{0}1.293$&\\
&forests & $\phantom{0}2.9557\phantom{*}$ &\cite{Otter48}& $\phantom{0}1.626$&\\ 
&outer-planar & $\phantom{0}7.5036\phantom{*}$ &\cite{BFKV07}& $\phantom{0}3.275$&\\
&series-parallel & $\phantom{0}9.0733^*$ &\cite{BGKN07}& $\phantom{0}3.850$&\\ 
&$K^-_5$-minor-free & $15.65^*\phantom{00}$ &\cite{GN09a}& $\phantom{0}6.264$&\\
&planar & $27.2268^*$ &\cite{GN09a}& $10.520$&\\
&bounded genus & $27.2268^*$ &\cite{McDiarmid08}& $10.520$&\\ 
&$K_{3,3}$-minor-free & $27.2293^*$ &\cite{GGNW08}& $10.521$&\\
\hline
\end{tabular}
\caption{Lower bounds for various family $\sF$ of the form $\sU(\sF_n) \ge cn - o(n)$, where $c$ is the solution of $g = c^c/(c-1)^{c-1}$ and where $g$ is a lower bound on the unlabeled constant-growth of $\sF$. All figures in the numbers are correct (not rounded). Starred numbers correspond to the \emph{labeled} constant-growth of $\sF$ defined as $\limsup_{n\to\infty}{(\ell_n/n!)^{1/n}}$, where $\ell_n$ counts the number of labeled graphs of $\sF$ having $n$ vertices, a lower bound on the unlabeled constant-growth of $\sF$. }
\label{app:tab:lb_subgraphs}
\end{center}
\end{table}
\footnotetext{Improved by \cref{th:lb_cliques}.}

The labeled constant-growth of other families can be founded in~\cite{GN09a,McDK12,NRR20}. An intriguing question, we left open, is to know whether there are hereditary families $\sF$ with unlabeled constant-growth $g$ for which the lower bound given by \cref{th:lb_subgraphs} is tight.


\section{Almost Equitable Coloring}
\label{sec:almosteq}

The main idea in our construction of small induced-universal graph in \cref{th:universal} is to reuse large stable sets found in each graph of the family. For this purpose, one need these stable sets have similar size.

Recall that a $k$-coloring is \emph{equitable} if the number of vertices in two color classes differ by at most one. In other words, if $G$ has an equitable $k$-coloring, then $V(G)$ has a partition into stable sets $C_1,\dots,C_k$ such that $|C_i| \in \set{\ceil{n/k},\floor{n/k}}$ for each $i$, where $n = \cardV{G}$. Note that $|C_i| = 0$ for some $i$ implies that $|C_j| \in \set{0,1}$ for all $j$.

It is well-known that any $n$-vertex graph with maximum degree $\Delta$ has an equitable $(\Delta+1)$-coloring, as proved by Hajnal and Szemer{\'e}di (see~\cite{KK08a} for a shorter proof), that can be computed in time $O(\Delta n^2)$~\cite{KKMS10}. It is conjectured that in fact every connected graph is equitably $(\Delta+1)$-colorable~\cite{Meyer73}, if it is neither a clique nor an odd cycle. This has been confirmed for many graph classes. In particular, every $4$-colorable graph of maximum degree~4 has an equitable $4$-coloring~\cite{KK12a}.

We point out that there are graphs with equitable $k$-coloring but no equitable $(k+1)$-coloring, e.g. a $2$-coloring of $K_{3,3}$. There are also bipartite graphs with no equitable $2$-coloring, for instance $K_{p,q}$ with $q > p+1$. Even worst for $p=1$, $K_{1,q}$ has an equitable $k$-coloring if and only if $k\ge 1 + q/2$. However, by removing $p$ vertices, we can obtain a graph that has an equitable $2$-coloring, and even so an equitable $1$-coloring.

This leads to the following extension of equitable coloring that will be the crux for our main result. 
See \cref{app:fig:1-almost-equitable} for an illustration.

\begin{restatable}{definition}{defalmost}\label{def:almost_equitable}
  A graph $G$ is \emph{$p$-almost equitably $k$-colorable} if there exists a set $X$ of at most $p$ vertices such that $G \setminus X$ has an equitable $k$-coloring. 
\end{restatable}

\begin{myfigure}
  \tikzstyle{mypoint}=[inner sep=1.4pt, minimum size=7pt]
  \tikzstyle{myrect}=[minimum width=0pt,minimum height=0pt]
  \tikzset{green_node/.append style={mypoint}}
  \tikzset{red_node/.append style={mypoint}}
  \tikzset{red_square/.append style={myrect}}
  \tikzset{blue_square/.append style={myrect}}
  \tikzset{green_square/.append style={myrect}}
  \ctikzfig{1-balanced-cat_2}
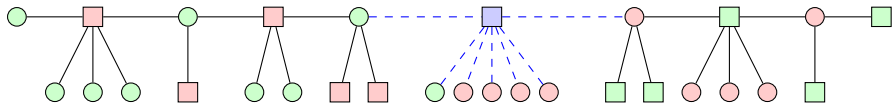
  \caption{A caterpillar of~$29$ vertices with a $2$-coloring that is not equitable (18 circle vs. 11 squared vertices). In fact, it has no equitable $2$-coloring, but it has a $1$-almost equitable $2$-coloring (14 green and 14 red vertices) obtained by removing the squared vertex in the meddle of the path.}
  \label{app:fig:1-almost-equitable}
\end{myfigure}

So, equitably $k$-colorable graphs, i.e., the graphs having an equitably $k$-coloring, are exactly those that are $0$-almost equitably $k$-colorable. Clearly, every graph $G$ with $n$ vertices is $n$-almost equitably $k$-colorable, $G\setminus X$ has only $k$ vertices which can be (equitably) $k$-colored anyway.
It is also straightforward to see that:

\begin{claim}\label{claim:equal}
If $G$ has $n$ vertices and is $p$-almost equitably $k$-colorable with $k\le n$, then there exists a set $X$ with exactly $\min\set{p,n-k}$ vertices such that $G\setminus X$ has an equitable $k$-coloring.
\end{claim}

Indeed, if $|X| < \min\set{p,n-k}$, one can successively increase $X$ set by taking one vertex at each step from the stable set with the current most frequent color. We will use \cref{claim:equal} twice in this section.

\subsection{Block design based construction}
\label{sec:block_design}

\cref{th:universal} below relies on the well-known Coding-Theoretic Function $A(n,d,w)$ that is the maximum number of binary words of $n$ bits that are pairwise at distance\footnote{I.e., the Hamming distance.} at least $d$ and of weight\footnote{I.e., the number of 1's in the binary word.} $w$. For instance, with the bijection between binary words of length $n$ to subsets of $\range{1}{n}$, $A(n,2,2)$ can be seen as the number of distinct pairs of integers taken from $\range{1}{n}$, since to be at distance at least two, binary words of weight two can share at most one~1. Therefore, $A(n,2,2) = \binom{n}{2}$. More generally, for each integer $k\ge 2$, $A(n,2k-2,k)$ is the maximum number of edge-disjoint copies of $K_k$ taken in $K_n$. Such clique partitions are also called \emph{block designs} or \emph{incomplete block designs}, and are parts of Combinatorial Designs field~\cite{CD07}.

Not all the values of $A(n,d,w)$ are known, and we refer to~\cite{BSSS90,CD07} for best known bounds on $A(n,d,w)$. However, by counting edges of the cliques, it is easy to see that $A(n,2k-2,k) \le \Bdiv{n}{k}$. The equality holds if and only if a Steiner System $S(2,k,n)$ exists (for instance, see~\cite[Theorem~7]{BSSS90}). E.g., the equality holds for any $a \in\mathbb{N}$, for any prime power $k$ and $n = k^a$, or, for any $k$ power of two and $n = (k-1) 2^{a} + k$ (cf.~\cite[Theorem~13, Eq.~(26{\&}29)]{BSSS90}.

By combining block designs and almost equitable coloring, we can show:

\begin{restatable}{theorem}{thuniversal}\label{th:universal}
  Let $s,t,k \in\mathbb{N}$ such that $t\le A(s,2k-2,k)$. For every family $\sF$ with $t$ graphs, each being $p$-almost equitably $k$-colorable with $n$ vertices,
  \[
    \sU(\sF) ~\le~ s \ceil{\frac{n-p}{k}} + t p ~.
  \]
\end{restatable}

Given the numbers $A(s,2k-2,k)$, the induced-universal graph $U_{\sF}$ for $\sF$ has an explicit construction. Furthermore, all embedded graphs of $\sF$ in $\sU_{\sF}$ are pairwise edge disjoint, as only vertices are shared into some common stable sets. It follows that $U_{\sF}$ is also an induced-universal graph for any edge-colored graph of $\sF$. For the family of all edge-colored cliques with $r$ colors, we refer to~\cite{Kouekam21}.

\medskip

\begin{proofof}\textbf{\cref{th:universal}.}
  Consider a family $\sF$ of $t$ graphs with $n$ vertices that are $p$-almost equitably $k$-colorable, where $t \le A(s,2k-2,k)$ for
  some integer $s$.

  We define the graph family $\cH$ obtained from $\sF$, by removing, for each $G\in\sF$, exactly $p$ vertices of $G$ such that the resulting graph has an equitable $k$-coloring. The fact we can remove exactly $p$ vertices is due to \cref{claim:equal}. By construction, $|\cH| \le |\sF| = t$.

  $U_{\cH}$ is composed of $s$ stable sets $S_1\dots,S_s$, each with $\ceil{(n-p)/k}$ vertices.

  We define $\sK_s$ be the complete graph on $s$ vertices that are the stables $S_1,\dots,S_s$. As we already said, it is possible to pack
  $A(s,2k-2,k)$ cliques of $k$ vertices in $\sK_s$ in a way that no two cliques of the packing share more than one vertex. Let
  $Q_1,\dots,Q_t$ be such clique packing in $\sK_s$. Note that this is possible because $t\le A(s,2k-2,k)$. E.g., if $k = 2$, $Q_i$'s are
  just $K_2$.

  The edges of $U_{\cH}$ are determined by the graphs of $\cH$ as
  follows. We associated with each graph $H_i \in \cH$ a unique clique
  $Q_i$ of the packing. In $U_{\cH}$, $Q_i$ corresponds to a
  collection of $k$ stables $C_1,\dots,C_k$, taken among the
  $S_i$'s. In the meanwhile, the equitable $k$-coloring of $H_i$
  induced a partition of $V(H_i)$ into $k$ stables, each with at most
  $\ceil{(n-p)/k}$ vertices. We map all vertices of a given part of
  $H_i$, say those of color $j$, to some stable $C_j$ in
  $U_{\cH}$. Then, we add to $U_{\cH}$ all edges of $H_i$, so that
  are between some $C_j$ and $C_{j'}$. Mapping the vertices of color
  $j$ in $H_i$ on $C_j$ is $U_{\cH}$ is possible because each $C_j$
  has $\ceil{(n-p)/k}$ vertices.

  From the above construction, one can check that each $H_i \in\cH$ is
  an induced subgraph of $U_{\cH}$, since any two graphs of $\cH$
  share at most one stable in $U_{\cH}$. It shows that $U_{\cH}$ is an
  induced-universal graph for $\cH$ with $s\ceil{(n-p)/k}$ vertices.

  The graph $U_{\sF}$ is obtained from $U_{\cH}$ by adding the
  following set of vertices and edges. For each $G_i \in\sF$, we select
  some $H_i = G_i\setminus X_i$ in $\cH$ with $|X_i| = p$. This is possible because $G$ is $p$-almost equitably $k$-colorable. The graph
  $H_i \in\cH$ appears in $U_{\cH}$ as induced subgraph. Then, we add
  $p$ vertices that are connected to $H_i$ in the same way $X_i$ is
  connected to its neighbors in $G_i$ as shown in \cref{fig:univ}. 
  Repeating the same process for each graph $G_i\in \sF$ adds $tp$ 
  vertices.

  Eventually, the graph $U_{\sF}$ has $s\ceil{(n-p)/k} + tp$ vertices
  and is induced-universal for the family $\sF$.  See \cref{fig:univ}
  for an illustration with $k=2$.

  \begin{myfigure}
    \tikzstyle{mypoint}=[inner sep=1.4pt, minimum size=7pt]
    \tikzset{green_node/.append style={mypoint}}
    \tikzset{red_node/.append style={mypoint}}
    \tikzset{blue_node/.append style={mypoint}}
    \tikzset{orange_node/.append style={mypoint}}
    \tikzset{pink_node/.append style={mypoint}}
    \tikzset{circle/.append style={mypoint}}
    \ctikzfig{small_universal_graphs_construction}
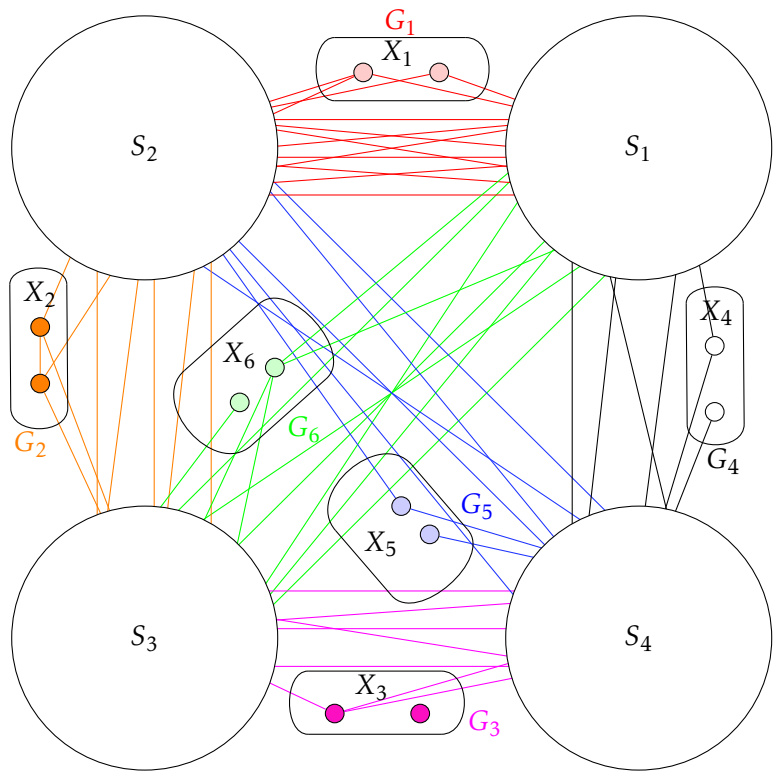
    \caption{Construction of $U_{\sF}$ for a family
      $\sF = \set{G_1,\dots,G_6}$ of $2$-almost $2$-colorable graphs.
      Here $p = k=2$, $s = 4$, and
      $t = A(s,2k-2,k) = A(s,2,2) = \binom{s}{2} = 6$. The stables
      $S_1,\dots,S_4$, each of size $\ceil{(n-p)/k}$, are shared
      between the $6$ graphs of $\sF$. }
    \label{fig:univ}
  \end{myfigure}
\end{proofof}


The above construction relies on the function $A(s,2k-2,k)$, whose values are not all known, but can be computed in practice (see for instance~\cite{Owen95}).  

\subsection{Bounded pathwidth graphs}

We make a link between pathwidth and almost equitable coloring. Recall that a graph $G$ has pathwidth at most $p$ if it is a subgraph of some interval graphs with maximum clique size $p+1$.

\begin{restatable}{theorem}{thpathwidth}\label{th:pathwidth}
  Every $k$-colorable graph of pathwidth at most $p$ is $p(k-1)$-almost equitably $k$-colorable.
\end{restatable}

Because graphs of pathwidth $p$ are $(p+1)$-colorable (as they are $p$-degenerated), we immediately get from \cref{th:pathwidth}:

\begin{corollary}\label{cor:3pw_sqrt}
  Every graph of pathwidth at most~$p$ is $p^2$-almost equitably $(p+1)$-colorable.
\end{corollary}

The strategy to prove \cref{th:pathwidth} is as follows. 

First, we select two colors that are non-equitable for the $k$-coloring, the least and the most frequent colors. Second, we show how to re-balance these two colors, by removing at most $p$ vertices, so that one of the color has $n/k - p(k-1)/k$ vertices. Then, we set this color aside and repeat this process $k-1$ times in total, until all the colors are balanced, so by removing a total of $p(k-1)$ vertices.

For the second part, we need of the following technical lemma.

\begin{lemma}\label{lem:rebalanced}
  Let $B$ be a graph having a path-decomposition of width $p$ and a $2$-coloring into stables $C_1, C_2$ with $|C_1| \ge |C_2|$. Then,
  for every non-negative integer $a \in [|C_2|-p, |C_1|]$ there exists a set $X$ of at most $p$ vertices and a $2$-coloring $C'_1,C'_2$ of $B\setminus X$ such that $|C'_1| = a$.
\end{lemma}

\begin{proof}
  Let $B$ be a bipartite graph, and $P = (X_1, \ldots, X_r )$ be a path-decomposition of width $p$. According  to~\cite[Lemma~7.2]{CFKLx15}, we can assume that $P$ is \emph{nice}, meaning that $X_1,X_r$ are empty, and that $X_{i+1}$ is obtained from $X_i$ either by adding a vertex not in $X_i$, or by removing a vertex from $X_i$. Some $X_i$'s other than $X_1$ and $X_r$ may be empty, in particular if $B$ is not connected.
  
  For every $X_i$ of $P$, we define a $2$-coloring $C_1^i,C_2^i$ of $B\setminus X_i$ as follows. Informally, all the vertices of
  $B\setminus X_i$ \emph{after} $X_i$ in $P$ (i.e., every $u \in X_j\setminus X_i$ with $j > i$) keep their color. Moreover, the ones \emph{before} $X_i$ (i.e., every $u \in X_j\setminus X_i$ with $j < i$) exchange their colors. More precisely, we define $C_1^i$ as
  the set of vertices of $B\setminus X_i$ that are in $C_1$, if after $X_i$, and in $C_2$, if before $X_i$. Moreover, define $C_2^i = V(B)\setminus X_i \setminus C_1^i$, that is the set of vertices of $B\setminus X_i$ that are in $C_2$, if after $X_i$, and in $C_1$, if before $X_i$. Note that, for each $i$, $C_1^i,C_2^i$ is a $2$-coloring of $B\setminus X_i$ since, from the path-decomposition $P$, vertices before and after $X_i$ are in different connected components of $B\setminus X_i$. Since there no vertices before $X_1$ and $B\setminus X_1 = B$, we have $C_1^1 = C_1$ and $C_2^1 = C_2$. Similarly, there no vertices after $X_r$ and $B\setminus X_r = B$, so $C_1^r = C_2$ and $C_2^r = C_1$.

  For simplicity, denote by $a_i = |C_1^i|$ and $b_i = |C_2^i|$ the color's counters for $C_1^i,C_2^i$. The $2$-colorings defined for $i=1$ and $i=r$ have their colors inverted. So, $a_1 = |C_1^1| = |C_1| = b_r$ and $b_1 = |C_2^1| = |C_2| = a_r$. Note that $a_1 \ge a_r$, since $|C_1| \ge |C_2|$.

  Consider any non-negative integer $a \in [|C_2|-p, |C_1|]$. In other words, $a \in [\max(0,a_r-p), a_1]$. If $a \le a_r$, then we can select for $X$ any subset of $C_2$ with $|X| = a_r - a$ vertices. Since $a \ge \max(0,a_r-p)$, then $a_r - a \le a_r - \max(0,a_r - p) = \min(a_r - 0, a_r - (a_r-p)) \le p$, and thus we have $|X| \le p$, and $|C_2\setminus X| = |C_2| - |X| = a_r - (a_r - a) = a$. Thus, if $a \le a_r$, we are done with the initial 2-coloring of $B$ by setting $C'_1 = C_2 \setminus X$ and $C'_2 = C_1\setminus X = C_1$. So, from now on, let us assume $a \in (a_r,a_1]$.

  Two consecutive bags of $P$ change by exactly one vertex (recall that $P$ is nice). Therefore, whenever moving from $X_i$ to $X_{i+1}$ along $P$, the color's counters for the 2-coloring $C_1^i,C_2^i$ goes from $(a_i,b_i)$ to $(a_{i+1},b_{i+1})$ for $C_1^{i+1},C_2^{i+1}$ such that
  \[
    (a_{i+1},b_{i+1}) \in \set{ (a_i - 1,b_i), (a_i + 1, b_i), (a_i,
      b_i - 1), (a_i, b_i + 1) } ~.
  \]
  Since $a \in (a_r,a_1]$, there must exist at least one bag $X_j$ of $P$ such that $a_j = a$ since the color's counter moves decreasingly one-by-one on each component from $(a_1,b_1)$ to $(a_r,b_r)$ with $a_1 \ge a_r$. W.l.o.g. assume $j$ is the largest index such that $a_j = a$.

  If $|X_j| \le p$, then we are done by selecting $X = X_j$ and the 2-coloring $C'_1 = C_1^j$ and $C'_2 = C_2^j$. We will conclude now by showing that the case $|X_j| = p + 1$ is not possible.

  Assume $|X_j| = p+1$. Consider $X_{j+1}$, which exists in $P$, since $|X_j| > 0$ and thus $X_j \neq X_r = \emptyset$. Note that $|X_{j+1}| < |X_j|$, since $P$ is nice. We also remark that $a_i + b_i + |X_i| = \cardV{B}$ for every $i$. If $a_{j+1} = a$, then $j$ is not maximal: a contradiction. If $a_{j+1} < a$, then $b_{j+1} = b_j$. It follows that $a_{j+1} + b_{j+1} + |X_{j+1}| < a_j + b_j + |X_j|$: a contradiction, both sums must be $\cardV{B}$. So, we are left with the case $a_{j+1} > a$. Thus, we have $a_r < a < a_{j+1}$. By considering the bags of $P$ from $X_{j+1}$ to $X_r$, and the color counters from $(a_{j+1},b_{j+1})$ to $(a_r,b_r)$, there must be an index $j' \in [j+1,r]$ such that $a_{j'} = a$: a contradiction with the maximality of $j$.
\end{proof}


\begin{proofof}\textbf{\cref{th:pathwidth}.}
  Let $G$ be a graph with $n$ vertices having a path-decomposition of width $p$ and a $k$-coloration into stable sets $C_1, \ldots, C_k$. The willing $k$-coloring for $G$ is obtained by applying the following process (initially the current graph is $G$ it-self):

  \begin{enumerate}

  \item Select in the current graph the two stable sets corresponding to the least and most frequent colors, say $C_{\min},C_{\max}$.

  \item Apply \cref{lem:rebalanced} on the bipartite subgraph induced by $C_{\min},C_{\max}$, and with a suitable non-negative integer $a \in [|C_{\min}|-p,|C_{\max}|]$. This provides a set $X$ and a new $2$-coloring $C'_1,C'_2$ with $|C'_1| = a$ and $|X|\le p$. Adjust $X$ and $C'_2$ such that $|X| = p$.

  \item Update the current graph by moving aside vertices of $X$ and
    $C'_1$, and repeat.

  \end{enumerate}
  
  It is clear that after repeating this process $k-1$ times, we obtain a new $k$-coloring for $G$, the last color being composed of the remaining graph. The first $k-1$ colors are all the $C'_1$'s stable sets constructed during the process. Moreover, the deleting set is the union of all the $X$'s sets constructed. It contains exactly $p(k-1)$ vertices.

  What we need to check is that such a $k$-coloring for $G$, without the deleting set, is equitable.

  To be more precise, we introduce the following notations. Consider the situation obtained after applying $i$ loops of the above process,
  and denote by:

  \begin{itemize}[noitemsep]

  \item $G_i$ the current graph, initially $G_0 = G$;

  \item $C_1^i,\dots,C_{k-i}^i$ the current $(k-i)$-coloring of $G_i$,
    initially $(C_1^0,\dots,C_{k}^0) = (C_1,\dots,C_k)$.

  \item $C'_1,\dots,C'_i$ the current sequence of $i$ colors that are supposed to be equitable, initially, for $i=0$, the sequence is empty, no such colors have yet been defined; and
 
  \item $X_i$ the current deleting set for $G$, obtained by taking the union of all $X$'s sets constructed so far, initially $X_0 = \emptyset$.
    
  \end{itemize}
  
  \def\state#1#2#3{\langle~{#1}~|~{#2}~|~{#3}~\rangle}

  After $i$ loops, the current state can be summarized as a partition
  of the vertex set of $G$, and can be represented by
  \[
    \state{C_1^i,\dots,C_{k-i}^i}{X_i}{C'_1,\dots,C'_i}
  \]
  where $G_i$ is the graph induced by $C_1^i\cup\cdots\cup
  C_{k-i}^i$. Initially, for $i=0$, we have the state
  \[
    \state{C_1,\dots,C_{k}}{\emptyset}{\eps}
  \]
  where $\eps$ denotes the empty sequence.

  So, the $(i+1)$-th step of the process can be rephrased as follows:

  \begin{enumerate}

  \item Select in $G_i$ two stables set $C^i_{\min},C^i_{\max}$, among its
    $k-i$ colors $C_1^i,\dots,C_{k-i}^i$, with least and most
    cardinality.

  \item Apply \cref{lem:rebalanced} on the induced subgraph
    $G[C^i_{\min} \cup C^i_{\max}]$ with a suitable parameter $a_i$. This
    provides a set $X$ and a stable set $C'_i$ with $|C'_i| =
    a_i$. Adjusting $X$ to $p$ vertices, we form $X_i$ with $i\cdot p$
    vertices.

  \item Update $G_i$ and its coloring $C_1^i,\dots,C_{k-i}^i$ into
    $G_{i+1}$ and $C_1^{i+1},\dots,C_{k-i-1}^{i+1}$.

  \end{enumerate}
  It can also be illustrated by the state transition:
  \[
    \state{C_1^i,\dots,C_{k-i}^i}{X_i}{C'_1,\dots,C'_i}
    \quad\xrightarrow{~i+1~}\quad
    \state{C_1^{i+1},\dots,C_{k-i-1}^{i+1}}{X_{i+1}}{C'_1,\dots,C'_i,C'_{i+1}}
  \]

  After $k-1$ steps, we have a deleting set $X_{k-1}$ with exactly
  $p(k-1)$ vertices, and a $k$-coloring of $G\setminus X_{k-1}$ into
  stable sets $C'_1,\dots,C'_{k-1}$ and $C_{1}^{k-1}$ (the remaining
  stable set, actually $G_{k-1}$). In other words, the last transition is:
  \[
    \state{C_1^{k-2},C_2^{k-2}}{X_{k-2}}{C'_1,\dots,C'_{k-2}}
    \quad\xrightarrow{~k-1~}\quad
    \state{C_1^{k-1}}{X_{k-1}}{C'_1,\dots,C'_{k-1}}
  \]

  To show that these colors are equitable for $G\setminus X_{k-1}$, we have to show that $a_i$'s can be chosen such that $a_i \in \set{\ceil{(n-|X_{k-1}|)/k},\floor{(n-|X_{k-1}|)/k}}$, and, similarly the for the last color $C_1^{k-1}$, that $C_1^{k-1} \in \set{\ceil{(n-|X_{k-1}|)/k},\floor{(n-|X_{k-1}|)/k}}$. Furthermore, in order to apply \cref{lem:rebalanced}, we have to check that
  $a_i \in [|C^i_{\min}|-p,|C^i_{\max}|]$ and $a_i \ge 0$.

  We remark that for non-null integers $m,k$, we have $m = k\floor{m/k} + r$, where $r = m \bmod k$. Observing that
  $\ceil{m/k}-\floor{m/k} \in \set{0,1}$ depending whether $k$ divides
  $m$ or not, this can be rewritten as $m = (k-r)\floor{m/k} + r \ceil{m/k}$.

  So, to be equitable, we must have in $G\setminus X_{k-1}$ precisely
  $k-r$ stable sets of size $\floor{(n-|X_{k-1}|)/k}$ and $r$ stable sets of
  size $\ceil{(n-|X_{k-1}|)/k}$, where $r = (n-|X_{k-1}|) \bmod
  k$. Since $n-|X_{k-1}| = n - p(k-1)$, we have $r = (n+p) \bmod k$,
  and $(n-|X_{k-1}|)/k = (n+p)/k - p$. So, we must have $r$ stable sets of
  size $\ceil{(n+p)/k}-p$ and $k-r$ stable sets of size
  $\floor{(n+p)/k}-p$, including the last color $C_1^{k-1}$.

  For the $r$ first steps $i = 1,\dots,r$, we choose
  $a_i = \ceil{(n+p)/k} - p$, and for $i>r$ we choose
  $a_i = \floor{(n+p)/k} - p$. As $a_i$'s differ by at most one, it
  will be more convenient to rewrite $a_i$ as
  \[
    \forall i \in\range{1}{k-1}, \quad a_i ~=~ \floor{\frac{n+p}{k}} - p +
    d_i \quad\mbox{, where~} d_i = \left\{%
      \begin{array}{ll}
        1 & \mbox{if $r>1$ and $i\le r$}\\
        0 & \mbox{if $r=0$ or $i>r$}
      \end{array}\right.%
  \]
  It is easy to check that $\sum_{j=1}^i d_i = \min(i,r)$, since up to
  index $r$, $d_i = 1$, and then $d_i = 0$ (we have always $d_i= 0$ if
  $r = 0$). W.l.o.g. we assume that $n \ge p(k-1)$, since
  \cref{th:pathwidth} is clearly true otherwise. It follows that
  $(n+p)/k - p \ge 0$, and thus all the $a_i$'s are $\ge 0$.

  We assume that \cref{lem:rebalanced} applies correctly for the first
  $i$ steps of the process, that is
  $a_i \in[|C^i_{\min}|-p,|C^i_{\max}|]$. Let $n_i = \cardV{G_i}$ be the
  number of vertices in the current graph $G_i$. By construction,
  $n_i = n - |X_i| - \sum_{j=1}^i |C'_i|$. Since \cref{lem:rebalanced}
  applies, $|C'_i| = a_i$, and we have
  \begin{eqnarray*}
    n_i &=& n - |X_i| - \sum_{j=1}^i a_i\\
        &=& n - ip - \sum_{j=1}^i \pare{ \floor{\frac{n+p}{k}} - p + d_j }\\
        &=& n - i \floor{\frac{n+p}{k}} - \sum_{j=1}^i d_j\\
        &=& n - i \floor{\frac{n+p}{k}} - \min(i,r) ~.
  \end{eqnarray*}

  We consider now Step $i+1$, for some $i \in \range{0}{k-2}$. There
  are $k-i$ colors in $G_i$, which has $n_i$ vertices, so the least
  and most frequent colors in $G_i$ satisfy
  $|C^i_{\min}| \le n_i/(k-i) \le |C^i_{\max}|$. To show that
  $a_{i+1} \in [|C^i_{\min}|-p, |C^i_{\max}|]$ (we have seen that
  $a_{i+1}$ -- in fact all the $a_i$'s -- is non-negative), it suffices to show that
  \begin{eqnarray}
    && \frac{n_i}{k-i} -p ~\le~ a_{i+1} ~\le~ \frac{n_i}{k-i} \nonumber\\
    &\Leftrightarrow& \frac{n_i}{k-i} - a_{i+1} \in [0,p] \nonumber\\
    &\Leftrightarrow& n_i - (k-i) \cdot a_{i+1} \in I, \mbox{~where~}I
                      = [0,(k-i)p] \nonumber\\
    &\Leftrightarrow& \pare{ n-i \floor{\frac{n+p}{k}} - \min(i,r) } -
                      (k-i) \pare{ \floor{\frac{n+p}{k}} - p +d_{i+1}
                      } \in I \nonumber\\
    &\Leftrightarrow& n - k \floor{\frac{n+p}{k}} - \min(i,r) + (k-i)
                      \pare{ p  - d_{i+1} } \in I \nonumber\\
    &\Leftrightarrow& r-p - \min(i,r) + (k-i)
                      \pare{ p  - d_{i+1} } \in I \label{eq:r-p}
  \end{eqnarray}
  the last simplification in Eq.~\eqref{eq:r-p} comes from the fact
  that we have chosen $r$ such that $n+p = k\floor{(n+p)/k} + r$.

  To prove Eq.~\eqref{eq:r-p}, we consider two cases.

  \paragraph{Case $d_{i+1} = 1$.} In that case $r>0$ and
  $i+1\le r$. So $\min(i,r) = i$. Eq.~\eqref{eq:r-p} can be simplified as:
  \begin{eqnarray*}
    && r - p - i + (k-i) (p - 1) \in I \\
    &\Leftrightarrow& (k-i-1)p - (k-r) \in I = [0,(k-i)p]\\
    &\Leftrightarrow& (k-i-1)p \ge (k-r) \mbox{~~and~~} (k-i-1)p - (k-r)
                      \le (k-i)p\\
    &\Leftarrow& (k-i-1)p \ge (k-i-1) \mbox{~~and~~} - (k-r)
                 \le p
  \end{eqnarray*}
  which is always true as $p \ge 0$ and $r < k$. Thus, Eq.~\eqref{eq:r-p} is true in this case.

  \paragraph{Case $d_{i+1} = 0$.} In this case, $r = 0$ or
  $i+1 > r$. We observe that in both cases, $r - \min(i,r) = 0$.  So,
  Eq.~\eqref{eq:r-p} can be simplified in
  \begin{eqnarray*}
    &&r - \min(i,r) -p + (k-i) (p - 0) \in I \\
    &\Leftrightarrow& (k-i-1)p \in [0,(k-i)p]
  \end{eqnarray*}
  which is true since $p\ge 0$ and $i+1\le k-1$. Thus, Eq.~\eqref{eq:r-p} is true in this case as well.

  All together, we have shown that the first $r$ colors satisfy
  $|C'_1|,\dots,|C'_r| = \ceil{(n+p)/k} - p$, and that the $k-r-1$
  next colors satisfy
  $|C'_{r+1}|,\dots,|C'_{k-1}| = \floor{(n+p)/k} - p$. It is easy to
  check that the last color $C^{k-1}_1$ has $\floor{(n+p)/k}-p$
  vertices. Indeed, by the partition of the vertices of $G\setminus
  X_{k-1}$, we must have
  \begin{eqnarray*}
    && n - |X_{k-1}| ~=~ \sum_{i=1}^r |C'_i| + \sum_{i=r+1}^{k-1} |C'_i| ~+~
       |C^{k-1}_1| \\
    &\Leftrightarrow& n - p(k-1) ~=~ r\pare{\ceil{(n+p)/k} - p} + (k-r-1)
                      \pare{ \floor{(n+p)/k} - p} + |C^{k-1}_1| \\
    &\Leftrightarrow& n ~=~ r\ceil{(n+p)/k} + (k-r-1)
                      \floor{(n+p)/k} + |C^{k-1}_1|\\
    &\Leftrightarrow& n+p ~=~ r\ceil{(n+p)/k} + (k-r-1)
                      \floor{(n+p)/k} + (|C^{k-1}_1|+p)
  \end{eqnarray*}
  This implies that $|C^{k-1}_1| + p = \floor{(n+p)/k}$ by the choice
  of $r$. So, $|C^{k-1}_1| = \floor{(n+p)/k}- p$ as required.

  This completes the proof of \cref{th:pathwidth}.
\end{proofof}



We left open the question about the optimality of \cref{th:pathwidth}. However, we can show that, for each $k$, the pathwidth in \cref{th:pathwidth} is required. More precisely:

\begin{proposition}
  For each integers $k$ and $p$, there is a $k$-colorable graph of pathwidth at most $p$ for which every $q$-almost equitably $k$-coloration requires $q \ge p$. 
\end{proposition}

\begin{proof}
Let $G$ be the complete $k$-partite graph with $k-1$ "small" parts of roughly $p/(k-1)$ vertices and the with one "large" part with $L = 2p + 2k-1$ vertices. To be more precise, each small part has $\floor{p/(k-1)}$ or $\ceil{p/(k-1)}$ vertices such that their size sum up to exactly $p$ vertices. The graph $G$ is $k$-colorable (each part is monochromatic, with a distinct color), and of pathwidth at most $p$ (consider a path of $L$ bags $B_1 - B_2 - \cdots - B_L$ where each $B_i$ is composed of the $p$ vertices contained in the small parts plus the $i$th vertex of the large part).

Consider any $q$-almost equitably $k$-coloration of $G$, let $X$ be the set of vertices that makes this coloring equitable, and let $G' = G\setminus X$. Suppose $|X| < p$.

There must be one of the small parts with $t>0$ vertices remaining in $G'$, because their sizes sum up to $p$. Note that $t\le \ceil{p/(k-1)} < p/(k-1) + 1$. It follows that $t+1 < p/(k-1) + 2$, and thus $(t+1)(k-1) \le p + 2k-1$.

In the large part, it remains $\ell > (t+1)(k-1)$ vertices in $G'$, because $L - |X| = 2p + 2k-1 - |X| > p + 2k-1 \ge (t+1)(k-1)$. The colors of any small part cannot appear in the large part. So, to color the vertices of the large part, at most $k-1$ colors are available. To be equitable, each color classes must have $t$ or $t+1$ vertices: a contradiction with $\ell > (t+1)(k-1)$.
\end{proof}

We note that the proofs of that \cref{th:pathwidth} and
\cref{lem:rebalanced} are constructive. Assuming a with-$p$
path-decomposition and a $k$-coloring of the graph are given, the
proofs lead to polynomial time algorithm for constructing the deleting set and the $k$-coloring making the graph $p(k-1)$-almost equitably
$k$-colored.

\subsection{Conflicting family}

An application of the construction in \cref{th:universal} is that it can give a lower bound on the size of a conflicting family w.r.t. some $cn$-lower bound.

\begin{restatable}{theorem}{thconflictsize}\label{th:conflict_size}
  Let $\sF$ be a family of $k$-colorable graphs with $n$ vertices and pathwidth $o(n/k)$. If $\sU(\sF) \ge cn - o(n)$ for some constant $c$, then $\sF$ contains more than $A(\ceil{ck}-1,2k-2,k)$ graphs, for every large enough $n$.
\end{restatable}

\begin{proof}
    Let $s$ be the largest integer such that $s < ck$ and let $a = A(s,2k-2,k)$. We have $s = \ceil{ck}-1$. We want to show that $\sF$ has to contain $t > a$ graphs. By way of contradiction, assume that $t \le a$.
    
    Note that $a \le \Bdiv{s}{k} \le (s/k)^2 < c^2$ by the choice of $s$. By \cref{th:pathwidth}, every graph of $\sF$ is $p(k-1)$-almost $k$-colorable, where $p$ is an upper bound on the pathwidth of the graph. Therefore, by \cref{th:universal}, $\sU(\sF) \le s n/k + tp(k-1)$. As $t\le a < c^2$, the additive term $tp(k-1) < c^2 pk = o(n)$ as $c$ is constant and $p = o(n/k)$ by assumption. Thus, $\sU(\sF) \le s n/k + o(n)$.
    
    This is incompatible with the assumption that $\sU(\sF) \ge cn - o(n)$, since, for every large enough $n$, this implies that $\sU(\sF)/n \ge c - o(1)$, whereas we have $\sU(\sF) \le s n/k + o(n)$ implies $\sU(\sF)/n < c + o(1)$ by the choice of $s$.
    
    It follows that $t > a = A(s,2k-2,k)$ where $s = \ceil{ck}-1$ as claimed.
\end{proof}

To illustrate \cref{th:conflict_size}, consider any family $\sF_n$ of bipartite graphs with sublinear pathwidth. From \cref{th:conflict_size}, if $\sU(\sF_n) \ge 2.1n -o(n)$, then $\sF_n$ must contains more than $A(\ceil{2.1\cdot 2} - 1, 2\cdot 2 - 2, 2) = A(4,2,2) = 6$ graphs. In particular, any conflicting family w.r.t. a $2.1n$-lower bound for trees must contains at least $7$ trees.

Now, we can combine the lower bounds on $c$ collected from \cref{app:tab:lb_subgraphs}, with \cref{th:conflict_size}. By this way, we obtain in \cref{app:tab:conflicting_size} the lower bounds on the size of conflicting families. Note that all families considered therein have bounded chromatic number and sublinear pathwidth.


\begin{table}[htbp!]
\begin{center}
\renewcommand{\arraystretch}{1.05}
\begin{tabular}{cc||c|c||clc}
&$\sF$ & $c$& $t$ & $k$ & $A(\ceil{ck}-1,2k-2,k)$&\\
\hline
&forests & $\phantom{0}1.626$ & $4$ & $2$ &$A(3,2,2) = 3$&\\ 
&outer-planar & $\phantom{0}3.275$& $13$ & $3$& $A(9,4,3) = 12$&\\
&series-parallel & $\phantom{0}3.850$ & $17$ & $3$& $A(11,4,3) = 16$&\\ 
&$K^-_5$-minor-free\myfootnotemark{foot}{1} & $\phantom{0}6.264$ & $51$ & $4$ & $A(25,6,4) = 50$&\\
&planar & $10.520$ & $137$ & $4$ & $A(42,6,4) = 136$&\\
&$K_{3,3}$-minor-free & $10.521$ & $124$ & $5$ & $A(52,8,5) \in\set{123,124}$\myfootnotemark{foot}{2}&\\
\hline
\end{tabular}
\caption{Lower bounds on the size $t$ of a conflicting families w.r.t. a $cn$-lower bound.}
\label{app:tab:conflicting_size}
\end{center}
\end{table}


\section{Towards Super-Linear Lower Bounds}
\label{sec:towards}

The motivation of this part is mostly linked to the research of lower bounds for families of $n$-vertex graphs for which one suspect a super-linear lower bound, say $\lambda(n) \cdot n$ where $\lambda(n)$ is non-constant. We show that, for many of those families, namely those with small chromatic number and sublinear pathwidth, the size of any conflicting family w.r.t. a $\lambda(n) \cdot n$-lower bound cannot be linear in $\lambda(n)$. As shown in the next theorem, it must be at least quadratic.

\begin{restatable}{theorem}{thasuniversal}\label{th:asuniversal}
  Let $\sF$ be a family of $n$-vertex $k$-colorable graphs of pathwidth at most $p$ such that $p (k^2-1) \le n$. If the number of graphs in $\sF$ is $t\ge \max\set*{4k^2,811}$, then
  \[
    \sU(\sF) ~<~ \frac{15}{7} \cdot \sqrt{t} \cdot n ~.
  \]
\end{restatable}

A consequence of \cref{th:asuniversal} is that any $\Omega(n\log{n})$-lower bound proof for planar graphs (if it exists), and more generally for fixed minor-closed family of graphs, must use a conflicting family of at least $t = \Omega(\log^2{n})$ witnesses $n$-vertex graphs.

First, this is because graphs excluding a fixed minor have pathwidth $O(\!\sqrt{n}\,)$ \cite{AST90}, and constant chromatic number, due to bounded edge density \cite{Mader67}, so conditions are fulfilled to apply \cref{th:asuniversal}. Second, for such family $\sF$ with $t = o(\log^2{n})$ graphs, \cref{th:asuniversal} would imply that $\sU(\sF) = o(n\log{n})$.

In order to prove \cref{th:asuniversal} we need, as intermediate step, the following result:

\begin{restatable}{theorem}{thasuniversaltwo}\label{th:asuniversal2}
  Let $\sF$ be a family of $n$-vertex $p$-almost equitably $k$-colorable graphs. If the number of graphs in $\sF$ is $t\ge \max\set*{k^2,811}$, then
  \[
    \sU(\sF) ~<~ \frac{15}{14} \cdot \sqrt{t} \cdot \pare{n-p+k} + t p ~.
  \]
\end{restatable}

A consequence of \cref{th:asuniversal2}, is that an upper bound on $\sU(\sF)$ can be derived for any graph family $\sF$, independently from its colorability. However, unlike \cref{th:asuniversal2}, it relies to families having $t = \Omega(n^2)$ graphs with $n$ vertices.

\begin{restatable}{corollary}{corntwo}\label{cor:n2}
    Every family of $n^2\ge 811$ graphs with $n$ vertices has an induced-universal graph with less than $15 n^2 /7$ vertices.
\end{restatable}

\begin{proof}
    Every $n$-vertex graph is equitably $n$-colorable. So \cref{th:asuniversal2} can be applied to any family $\sF$ of $n$-vertex graphs (with $k=n$ and $p=0$). Assuming that the number of graphs in $\sF$ is $t = k^2 = n^2$ graphs, we get by \cref{th:asuniversal2}, that $\sU(\sF) < 15/14\cdot \sqrt{t} \cdot (n - 0 + n) + t\cdot 0 = 15n^2/7$ as claimed.
\end{proof}

The proof of \cref{th:asuniversal2} relies on the two technical following lemmas about $A(n,2k-2,k)$. Indeed, in the proof of \cref{th:asuniversal}, we will need to lower bound $A(n,2k-2,k)$ for \textit{every} $k$, and not only for prime power $k$ as in \cite[Theorem~13]{BSSS90}. Furthermore, as already explained above \cref{def:almost_equitable}, majoring $k$ does not necessarily keep the equitably property of colorings. An alternative to our lower bound in \cref{lem:A_sk}, would be the use of the well-known optimal asymptotic bound due to Erd\H{o}s and Hanani~\cite[Theorem~1]{EH63} (see also~\cite[Theorem~3.1]{AY05} for further refinements) saying that, for each $k\ge 2$,
\[
    \lim_{n \to\infty} \frac{A(n,2k-2,k)}{\Bdiv{n}{k}} ~=~ 1 ~.
\]
Unfortunately, this would introduced an uncontrolled conditions between $s$ and $k$ due to this limit. Nevertheless, based on \cite{Bose38,Bush52}, it is possible to construct \emph{orthogonal arrays} $\mathrm{OA}(s^2,k,s,2)$, a generalization of orthogonal Latin squares, for every prime power $s$ and $k\le s+1$. From them, we can derive incomplete block designs showing that $A(sk,2k-2,k) \ge s^2$ (see \cite[Part~III]{CD07}). For completeness, we present another simple construction of such block designs, making the universal graphs given by \cref{th:asuniversal} and \cref{th:asuniversal2} fully constructive in polynomial time.

\begin{lemma}\label{lem:A_sk}
  Let $s,k \in\mathbb{N}$ such that $s$ is prime and $2 \le k \le s$. Then
\[
  A(sk,2k-2,k) ~\ge~ s^2 + k\cdot A(s,2k-2,k) ~.
\]
\end{lemma}

It is easy to check that if $A(s,2k-2,k) = \Bdiv{s}{k}$, then $s^2 + k\cdot A(s,2k-2,k) = s^2 + k\cdot \Bdiv{s}{k} = \Bdiv{sk}{k}$, and thus $A(sk,2k-2,k) = \Bdiv{sk}{k}$ from \cref{lem:A_sk}.

\medskip

\begin{proof}
Recall that $A(sk,2k-2,k)$ is the maximum number of edge-disjoint copies of $K_k$ taken in $K_{sk}$. To lower bound this number, we first split the $sk$ vertices of $K_{sk}$ into $k$ vertex-disjoint cliques $S_1,\dots,S_k$, each with $s$ vertices. Then, in each $S_i$, we select $A(s,2k-2,k)$ edge-disjoint copies of $K_k$. This is possible since each $S_i$ is a copy of $K_s$ and $2\le k\le s$. This first process creates a collection $\sC_1$ of $k \cdot A(s,2k-2,k)$ pairwise edge-disjoint copies of $K_k$ taken from $K_{sk}$.

In a second step, we will construct another collection $\sC_2$ of pairwise edge-disjoint copies of $K_k$, all obtained by picking one vertex in each $S_i$. By doing this way, cliques of $\sC_1$ and $\sC_2$ are edge-disjoint for sure, since clique-edges of $\sC_1$ involve endpoints inside a $S_i$'s, whereas clique-edges of $\sC_2$ have endpoints between distinct $S_i$'s. In total, this creates $|\sC_1| + |\sC_2|$ edge-disjoint copies of $K_k$ taken in $K_{sk}$, and thus it shows that $A(sk,2k-2,k) \ge |\sC_2| + k\cdot A(s,2k-2,k)$. See \cref{fig:A_sk} for an example.

It remains to show how to construct such a collection $\sC_2$ of edge-disjoint copies of $K_k$ such that $|\sC_2| = s^2$.

For convenience, denote $[t] = \range{0}{t-1}$ for any positive number $t$. Let $u_{i,j}$ be the $j$th vertices of $S_i$, for all $i\in [k]$ and $j\in [s]$. Then, for every $(p,q)\in [s]^2$, define $C_{p,q}$ be the clique of $K_{sk}$ induced by the vertex-set
\[
    V(C_{p,q}) ~=~ \set{ u_{i,j} : i \in [k] \mbox{~and~} j = p+iq \bmod s } ~.
\]

By construction, each $C_{p,q}$ is a clique on $k$ vertices. We set $\sC_2 = \set*{C_{p,q} : (p,q)\in [s]^2 }$. Clearly, $|\sC_2| = s^2$. It remains to show that cliques of $\sC_2$ are edge-disjoint, or equivalently, that any two cliques of $\sC_2$ share at most one vertex.

By way of contradiction, assume $u_{i,j}, u_{i',j'} \in V(C_{p,q}) \cap V(C_{p',q'})$ with $(i,j) \neq (i',j')$ and $(p,q) \neq (p',q')$.
We consider arithmetic in the cyclic group $\mathbb{Z}/s\mathbb{Z}$. Note that since $s$ is prime, $xy \equiv 0$ implies $x\equiv 0$ or $y\equiv 0$ for all integers $x,y$.


From $u_{i,j}, u_{i',j'} \in V(C_{p,q}) \cap V(C_{p',q'})$, we get the following equations on $j$ and $j'$:
\begin{eqnarray*}
  j  &\equiv~ p+iq   &\equiv~ p'+iq' \\
  j' &\equiv~ p+i'q &\equiv~ p'+i'q'
\end{eqnarray*}

If $i = i'$, then, using $j-j'\equiv (i-i')q$, we get $j-j' \equiv 0$, i.e., $j = j'$: a contradiction. Thus we have $i\neq i'$.

Using $j-j' \equiv (i-i')q \equiv (i-i')q'$, we get $(i-i')(q-q') \equiv 0$, that implies $i = i'$ or $q = q'$ since $s$ is prime. 
Since $i\neq i'$, we have $q = q'$.

Using $p + iq \equiv p' + iq'$, we get $p-p' \equiv (q'-q)i \equiv 0$ since $q = q'$. This implies $p - p' \equiv 0$, i.e., $p = p'$. This is a contradiction, since we have already $q = q'$.

Therefore, the two distinct vertices $u_{i,j}, u_{i',j'} \in V(C_{p,q}) \cap V(C_{p',q'})$ do not exist, and this completes the proof.
\end{proof}

\begin{myfigure}
    \begin{tikzpicture}
    \tikzset{every node/.style={vertex}}
    \def\radius{45mm} 
    \def\colors{{"red", "blue", "black", "orange", "olive"}} 
    \def\tt{360/30} 

    \foreach \i in {0,...,4} {
        \coordinate (x\i) at ({90+0*120+\tt*(\i-2)}:\radius) {};
        \coordinate (y\i) at ({90+1*120+\tt*(\i-2)}:\radius) {};
        \coordinate (z\i) at ({90+2*120+\tt*(\i-2)}:\radius) {};
    }
    
    \foreach \i in {0,...,3} {
        \pgfmathsetmacro{\j}{\i+1}
        \draw (x\i) -- (x\j);
        \draw (y\i) -- (y\j);
        \draw (z\i) -- (z\j);
    }
    \foreach \i in {x,y,z} {
        \draw (\i0) to [bend right=90] (\i2);
        \draw (\i2) to [bend right=90] (\i4);
    }
    
    \foreach \p in {0,...,4} {
    \foreach \q in {0,...,4} {
    \pgfmathsetmacro{\i}{int(mod(\p+0*\q,5))}
    \pgfmathsetmacro{\j}{int(mod(\p+1*\q,5))}
    \pgfmathsetmacro{\k}{int(mod(\p+2*\q,5))}
    \pgfmathsetmacro{\c}{\colors[\q]}
    \draw[\c] (x\i) -- (y\j) -- (z\k) -- cycle;
    }
    }

    \foreach \i in {0,...,4} {
        \node at (x\i) {};
        \node at (y\i) {};
        \node at (z\i) {};
    }

    \tikzset{every node/.style={draw=none}}

    \node at (x2) [above=8pt] {$S_1$};
    \node at (y2) [below left=8pt] {$S_2$};
    \node at (z2) [below right=8pt] {$S_3$};

    \node at (x0) [right=2pt] {$u_{0,0}$};
    \node at (x1) [above=.8pt] {$u_{0,1}$};
    \node at (x2) [right=2pt] {$u_{0,2}$};
    \node at (x3) [above=.8pt] {$u_{0,3}$};
    \node at (x4) [left=2pt] {$u_{0,4}$};

    \node at (y0) [above left=2pt] {$u_{1,0}$};
    \node at (y1) [left=2pt] {$u_{1,1}$};
    \node at (y2) [above left=2pt] {$u_{1,2}$};
    \node at (y3) [left=2pt] {$u_{1,3}$};
    \node at (y4) [below=2pt] {$u_{1,4}$};
    
    \node at (z0) [below=2pt] {$u_{2,0}$};
    \node at (z1) [right=2pt] {$u_{2,1}$};
    \node at (z2) [above right=2pt] {$u_{2,2}$};
    \node at (z3) [right=2pt] {$u_{2,3}$};
    \node at (z4) [above right=2pt] {$u_{2,4}$};
    

    \end{tikzpicture}
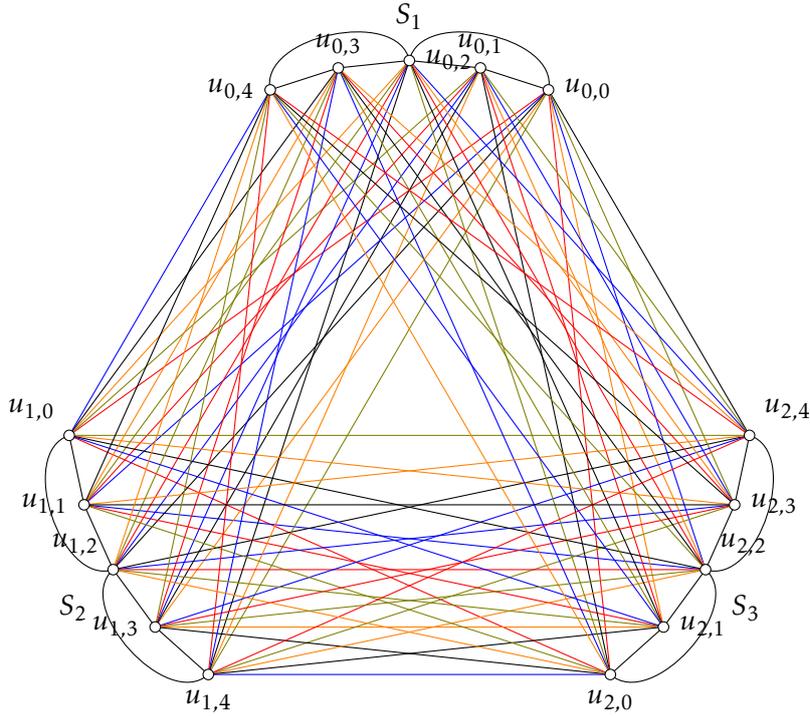
    \caption{Illustration of the block design construction as in \cref{lem:A_sk} for $s = 5$ and $k = 3$, aiming to lower bound $A(15,4,3)$. Note that $A(5,4,3) = 2$ (two $K_3$ sharing a vertex in $K_5$), and thus $A(sk,2k-2,k) \ge s^2 + k \cdot A(s,2k-2,k) = 25 + 3\cdot 2 = 31$. All the edge-disjoint $K_3$ (triangles) in $K_{15}$ are monochromatic, e.g. $V(C_{3,1}) = \set{u_{0,3},u_{1,4},u_{2,0}}$. The edges of $K_{15}$ are not represented. According to \cite[Table~I-A]{BSSS90}, $A(15,4,3) = 35$.}
  \label{fig:A_sk}
\end{myfigure}

\begin{lemma}\label{lem:A_n}
  For all $n,k \in\mathbb{N}$ such that $2\le k \le 14/15\cdot (n+1)/k$ and $30 \le n/k$,
  \[
    A(n, 2k-2, k) ~\ge~ \pare{\frac{14}{15} \cdot \frac{n+1}{k}}^2 ~.
  \]
\end{lemma}

\begin{proof}
We use a generalization of Betrand's postulate. From~\cite[Theorem~1]{SGM13}, for each integer $i\ge 2$, there is a prime in the interval $(14i,15i)$. Since $n \ge 30k$, the integer $i = \floor{n/(15k)} \ge 2$. Thus, there must exist a prime $s$ such that:
\[
    14\cdot \floor{\frac{n}{15k}} ~<~ s ~<~ 15\cdot \floor{\frac{n}{15k}} ~\le~ \frac{n}{k} ~.
\]
In particular, $n > sk$. Since $A(n,2k-2,k)$ is non-decreasing in $n$, we have $A(n,2k-2,k) \ge A(sk,2k-2,k)$. Note that\textsuperscript{\ref{foot:a/b}}
\[
    s ~\ge~ 14i+1 ~=~ 14 \cdot \floor{\frac{n}{15k}} + 1 = 14 \cdot\floor{\frac{n+15k}{15k}} ~\ge~ 14\cdot \pare{\frac{n+15k+1-15k}{15k}} ~=~ \frac{14}{15} \cdot \frac{n+1}{k} ~.
\]
Thus, since $k \le 14/15\cdot (n+1)/k$, we have $k\le s$. We can apply \cref{lem:A_sk} to get:
\[
    A(n,2k-2,k) ~\ge~ A(sk,2k-2,k) ~\ge~ s^2 ~\ge~ \pare{\frac{14}{15} \cdot \frac{n+1}{k}}^2
\]
as claimed.
\end{proof}

According to \cite{SGM13}, the constant $14/15$ that appears twice in \cref{lem:A_n} cannot be replaced by any ratio of the form $c/(c+1)$ as long as that the integer $c \le 100,000,000$. By using Betrand's postulate\footnote{It states that for each integer $i\ge 2$, there is a prime in the interval $(i,2i)$.}, we could achieve a better condition on $n/k$, namely $4\le n/k$, which ultimately would decrease\footnote{Following the proof of \cref{th:asuniversal2}, the new condition would be $t\ge \ceil*{((4+\nicefrac{1}{2})/2)^2} = 3$ instead of $t\ge \ceil*{((30+\nicefrac{1}{2})\cdot 14/15)^2} = 811$.} the constant~$811$ in \cref{th:asuniversal} and \cref{th:asuniversal2}. However, the price of weakening the condition is a doubling factor on the upper bound on $\sU(\sF)$, replacing the factors~$15/14$ and $15/7$ respectively by~$2$ and~$4$.

\medskip

\begin{proofof}\textbf{\cref{th:asuniversal2}.}
    We first observe that $k\ge 2$, since otherwise $\sF$ would be composed of only one graph (an independent set of $n$ vertices), contradicting the fact that $t \ge 811$.
    
    Our goal is to apply \cref{th:universal} to $\sF$. Let us show that if $s = \ceil{15/14 \cdot k\sqrt{t}} - 1$, then $t\le A(s,2k-2,k)$. By \cref{lem:A_n} (plugging $n=s$), $A(s,2k-2,k) \ge \pare{14/15 \cdot (s+1)/k}^2$, subject to $2\le k\le 14/15\cdot (s+1)/k$ and $30\le s/k$. So, it suffices to show that $14/15 \cdot (s+1)/k \ge \sqrt{t}$. This is equivalent to show that $s \ge 15/14 \cdot k\sqrt{t} - 1$. Thus, by choosing $s = \ceil{15/14\cdot k\sqrt{t}} - 1$, we ensure that $A(s,2k-2,k) \ge t$.

    The two conditions on $s$ and $k$ become:
    \[
    2 ~\le~ k ~\le~ \frac{14}{15}\cdot \frac{\pare{\ceil{15/14 \cdot k\sqrt{t}} - 1} + k}{k} \qquad\mbox{and}\qquad 30 ~\le~ \frac{\ceil{15/14 \cdot k\sqrt{t}} - 1}{k}
    \] 
    Recall that we have assumed $t \ge \max\set*{k^2,811}$. For the first condition, we have:
    \begin{eqnarray*}
        \frac{14}{15}\cdot \frac{\pare{\ceil{15/14 \cdot k\sqrt{t}} - 1}+k}{k} &\ge& \frac{14}{15}\cdot \frac{\pare{15/14 \cdot k\sqrt{t} + k-1}}{k} \\
        &\ge& \sqrt{t} + \frac{14}{15}\cdot\pare{1 - \frac{1}{k}} ~\ge~ \sqrt{t} ~\ge~ k
    \end{eqnarray*}
    as required.
    
    For the second condition, the term $s/k = \pare{\ceil{15/14\cdot k \sqrt{t}}-1}/k \ge 15/14\cdot \sqrt{t} - 1/k$, which is increasing with $k\ge 2$. So, the term $15/14\cdot \sqrt{t} - \nicefrac{1}{2}$ is at least $30$ whenever $\sqrt{t} \ge (30 + \nicefrac{1}{2}) \cdot 14 / 15 \approx \sqrt{810.35}$. Thus, since $t\ge 811$, we have that $s/k \ge 30$ as required.

    From \cref{th:universal},
    \begin{eqnarray*}
        \sU(\sF) &\le& s \ceil{\frac{n-p}{k}} + t p \\
    &<& \pare{\ceil{15/14 \cdot k\sqrt{t}} - 1} \pare{\frac{n-p+k}{k}} + t p \\
    &<& \pare{15/14 \cdot k\sqrt{t}} \pare{\frac{n-p+k}{k}} + t p \\
    &<& \frac{15}{14} \cdot \sqrt{t} \cdot \pare{n-p+k} + t p    \end{eqnarray*}
    completing the proof of~\cref{th:asuniversal2}.
\end{proofof}


\medskip

\begin{proofof}\textbf{\cref{th:asuniversal}.}
Similarly to the proof of \cref{th:asuniversal2}, we have $k\ge 2$ (otherwise in contradiction with $t \ge 811$). For technical reasons, we treat separately the case $p<2$ from the case $p\ge 2$. For the former case, if $p\in\set{0,1}$, the graphs of $\sF$ are caterpillar forests. We can use the construction of~\cite{GL22} showing that $\sU(\sF) \le 8n$, independently of $t$. As $t\ge 811$, $(15/7)\sqrt{t} \cdot n \ge (15/7)\sqrt{811} \cdot n \approx 61n > 8n \ge \sU(\sF)$. Therefore, the statement $\sU(\sF) \le (15/7)\cdot \sqrt{t} \cdot n$ is trivially true if $p<2$. So, w.l.o.g. we assume that $p\ge 2$.
    
    With each $G\in\sF$, we associate a new graph $G'$ such that:
        \begin{enumerate}[noitemsep]
            \item $G'$ contains $G$ as induced subgraph;
            \item $G'$ has $n' = 2(n-q)$ vertices, where $q = p(k-1)$; and
            \item $G'$ is equitably $2k$-colorable.
        \end{enumerate}
    The graph $G'$ is obtained from $G$ plus $n-2q$ isolated vertices. Note that $n-2q\ge 0$ since by assumptions $p(k^2-1) \le n$ and $k\ge 2$. Indeed, $2q = 2p(k-1) < p (2k-1) \le p (k^2-1) \le n$. 
    
    Clearly, $G'$ contains $G$ as induced subgraph and has exactly $n + n - 2q = n'$ vertices.
    
    To prove the third point, that is $G'$ is equitably $2k$-colorable, we partition $V(G')$ into two sets $S,T$, each of $n-q$ vertices, and show that $G'[S]$ and $G'[T]$ are both equitably $k$-colorable.
    
    To construct $S$, we applied to $G$ \cref{th:pathwidth} and \cref{claim:equal} to obtain a set $X$ of $|X| = \min\set{p(k-1),n-k}$ vertices such that $G\setminus X$ is equitably $k$-colorable. From the conditions on $(p,k,n)$, we get $p(k-1) + k \le p(k-1) + pk \le p(2k-1) \le p(k^2-1) \le n$. So $p(k-1) \le n-k$, and thus $|X| = p(k-1) = q$.
    
    We let $S = V(G)\setminus X$, and $T = V(G') \setminus S$. By construction, $G'[S] = G\setminus X$ has $n-q$ vertices and an equitable $k$-coloring into $k$ stable sets $C^S_1,\dots,C^S_k$ with $|C^S_i| \in\set{ \floor{(n-q)/k}, \ceil{(n-q)/k}}$ since we have seen that $q = |X|$.
    
    The graph $G'[T]$ is composed of $G[X]$ plus a set $I$ of $n-2q$ isolated vertices. Consider any $k$-coloring of $G[X]$ into $k$ stable sets $C^X_1,\dots,C^X_k$ (recall that $G$ is $k$-colorable). 
    We have $|C^X_i| \in \range{0}{|X|}$ for each $i$, noting that $|C^X_i| = 0$ is possible for some color $i$, in particular if $p=1$.
    
    The important remark is that the conditions on $(p,k,n)$ imply that $|C^X_i| \le |C^S_j|$ for all $i,j$. Indeed, we have $|C^X_i| \le |X| = q$ and $\floor{(n-q)/k} \le |C^S_j|$. Moreover, 
    \begin{align*}
            && q ~&\le~ \floor{(n-q)/k} &&\\
        \Leftrightarrow && q ~&\le~ (n-q)/k &&\\
        \Leftrightarrow && q(k+1) ~&\le~ n &&\\
        \Leftrightarrow && p(k-1)(k+1) ~&\le~ n &&\\
        \Leftrightarrow && p(k^2-1) ~&\le~ n ~. &&\\
    \end{align*}
    Now, $|C^X_i| \le |C^S_j|$ for all $i,j$ implies that we can color the vertices of $I$, and add $|C^S_i|-|C^X_i| \ge 0$ vertices taken from $I$ to $C^X_i$ resulting in a new color class of size exactly $|C^S_i|$. This process defines a $k$-coloring for $G'[T]$ whose color classes have the same color distribution as the $k$-coloring of $G'[S]$. Since $|S| = |T|$ and the $k$-coloring of $G'[S]$ is equitable, the $k$-coloring of $G'[T]$ must be equitable too.
    
    Now, we apply \cref{th:asuniversal2} to the family $\sF' = \set{ G' \colon G \in\sF}$ composed of all the graphs $G'$ constructed from $G \in\sF$ as above. Note that $\sF'$ is composed of $t$ graphs, each with $n'$ vertices, and that are equitably $k'$-colorable with $k'=2k$. Since $t\ge \max\set*{4k^2,811}$, we have that $t\ge \max\set*{k'^2,811}$ as required to apply \cref{th:asuniversal2}. Note that $k - q = k - p(k-1) \le k - 2(k-1) = 2-k \le 0$, because we have $k,p\ge 2$.
    
    Each graph $G\in \sF$ appears as induced subgraph of some $G' \in \sF'$. Therefore, by \cref{th:asuniversal2}, we get:
    \[
     \sU(\sF) ~\le~ \sU(\sF') ~<~ \frac{15}{14} \cdot \sqrt{t} \cdot \pare{n'+k'} ~\le~ \frac{15}{14} \cdot \sqrt{t} \cdot \pare{2(n - q + k)} ~\le~ \frac{15}{7} \cdot \sqrt{t} \cdot n ~.
    \]
    This completes the proof of \cref{th:asuniversal}.
\end{proofof}

\bibliographystyle{my_alpha_doi}

\bibliography{a}


\end{document}